\documentclass[10pt,a4paper]{amsart}

\usepackage{amssymb, url}
\usepackage{amsthm}
\usepackage{amsmath}
\usepackage[nodayofweek]{datetime}

\usepackage{enumerate}
\usepackage[colorlinks=true]{hyperref} 
\hypersetup{urlcolor=blue,linkcolor=black,citecolor=black,colorlinks=true}
\usepackage[all]{xy}
\usepackage{verbatim}
\usepackage{soul, dsfont}

\usepackage{xcolor,url}

\theoremstyle{plain} 
\newtheorem{Theorem}{Theorem}[section]
\newtheorem{Lemma}[Theorem]{Lemma}
\newtheorem{Corollary}[Theorem]{Corollary}
\newtheorem{Proposition}[Theorem]{Proposition}
\newtheorem*{Proposition*}{Proposition}
\newtheorem*{Corollary*}{Corollary}

\newtheorem{maintheorem}{Theorem}

\theoremstyle{definition}
\newtheorem{Definition}[Theorem]{Definition}
\newtheorem{Example}[Theorem]{Example}

\newtheorem*{Question*}{Question}
\newtheorem{Hypothesis}[Theorem]{Hypothesis}

\theoremstyle{remark}
\newtheorem{Remark}[Theorem]{Remark}

\newcommand{\CC}{\mathbb{C}}

\newcommand{\FF}{\mathbb{F}}
\newcommand{\GG}{\mathbb{G}}

\newcommand{\ZZ}{\mathbb{Z}}  
\newcommand{\NN}{\mathbb{N}}  

\newcommand{\cA}{\mathcal{A}}

\newcommand{\cM}{\mathcal{M}}
\newcommand{\cK}{\mathcal{K}}   
\newcommand{\cO}{\mathcal{O}}   
\newcommand{\ck}{\kappa_\fm}        
\newcommand{\cR}{\mathcal{R}}

\newcommand{\fp}{\mathfrak{p}}  
\newcommand{\fa}{\mathfrak{a}}  
\newcommand{\fm}{\mathfrak{m}}  

\newcommand{\sfA}{\mathfrak{A}}

\newcommand{\fK}{\mathfrak{K}}

\newcommand{\bA}{\mathbf{A}}

\newcommand{\bK}{\mathbf{K}}

\newcommand{\bsm}{\boldsymbol{m}}

\newcommand{\bsz}{\boldsymbol{z}}

\newcommand{\rd}{\mathrm{d}}

\newcommand{\oK}{\overline{K}}

\newcommand{\rP}[1]{\mathrm{P}_{\!#1}}

\newcommand{\Exp}{\operatorname{Exp}}
\newcommand{\Spec}{\operatorname{Spec}}

\newcommand{\tr}{\operatorname{tr}}

\newcommand{\Hom}{\operatorname{Hom}}
\newcommand{\End}{\operatorname{End}}
\newcommand{\Aut}{\operatorname{Aut}}
\newcommand{\id}{\operatorname{id}}
\newcommand{\Id}{\operatorname{Id}}
\newcommand{\Ker}{\operatorname{ker}}

\newcommand{\Fq}{\FF_{\!q}}  
\newcommand{\Fpp}{\FF_{\!\fp}} 

\newcommand{\Gal}{\operatorname{Gal}}

\newcommand{\Hde}{\operatorname{Hde}}

\newcommand{\Span}{\operatorname{Span}}
\newcommand{\rk}{\operatorname{rank}}
\newcommand{\trdeg}{\operatorname{tr.deg}}
\newcommand{\one}{\mathds{1}}

\newcommand{\res}{\operatorname{res}}

\newcommand{\Mat}{\mathsf{Mat}} 
\newcommand{\GL}{\operatorname{GL}}
\newcommand{\SL}{\operatorname{SL}}

\newcommand{\Lie}{\operatorname{Lie}}

\newcommand{\Gm}{\textbf{G}_{m}}

\newcommand{\CI}{\CC_{\infty}}  

\newcommand{\ps}[1]{[\![#1]\!]}  
\newcommand{\ls}[1]{(\!(#1)\!)}
\newcommand{\cs}[1]{\langle #1 \rangle} 
\newcommand{\mot}{\mathsf{M}}  
\newcommand{\dmot}{\mathfrak{M}} 
\newcommand{\up}{u_\fp}  
\newcommand{\sep}{\mathrm{sep}}
\newcommand{\image}{\operatorname{Im}}
\newcommand{\hd}[2]{\partial^{(#1)}\!\!\left(#2\right)} 
\newcommand{\hde}[1]{\partial^{(#1)}}  
\newcommand{\hdt}[2]{\partial_t^{(#1)}\!\!\left(#2\right)} 
\newcommand{\hdte}[1]{\partial_t^{(#1)}}  
\newcommand{\hdthe}[1]{\partial_{\theta}^{(#1)}}  
\newcommand{\D}{\mathcal{D}} 
\newcommand{\hdiff}[2]{\operatorname{d}^{(#1)}\left(#2\right)} 
\newcommand{\hdiffe}[1]{\operatorname{d}^{(#1)}} 
\newcommand{\power}[2]{{#1 [\![ #2 ]\!]}} 

\newcommand{\Cent}{\operatorname{Cent}}

\newcommand{\prol}[1]{\mathrm{p}_{[#1]}}

\newcommand{\fracTate}{\mathbb{L}}  

\numberwithin{equation}{subsection}

\title{Algebraic independence of periods of Anderson modules and their hyperderivatives}

\author{Andreas Maurischat}
\address{Andreas Maurischat, RWTH Aachen University, Aachen, Germany}

\author{Changningphaabi Namoijam}
\address{Changningphaabi Namoijam, Department of Mathematics, Fairfield University, Connecticut, USA}

\thanks{}

 \newdate{date}{21}{09}{2026}
 \date{\displaydate{date}}

\begin{document}

\begin{abstract} 
Transcendence questions of periods and quasi-periods of Anderson modules, as well as the hyperderivatives thereof, are of major interest in number theory over function fields.
In the case of Drinfeld modules, this question was answered by the second author.
In this paper, we show that these results hold for general Anderson $t$-modules under certain assumptions on their Galois representations.
The proofs use the explicit
description of the $\fp$-adic Galois representation of the first author by means of a rigid analytic
trivialization of the associated $t$-motive, as well as prolongations of $t$-modules and $t$-motives. The
general result is applied to determine algebraic relations between periods and quasi-periods, and their
hyperderivatives, of finitely many Drinfeld modules.
Along the way, our results provide more evidence for the Mumford-Tate conjecture for Anderson modules.
\end{abstract}

\maketitle
    

\setcounter{tocdepth}{1}
\tableofcontents

\section{Introduction}

\subsection{Motivation} 
Periods and quasi-periods of Drinfeld modules and their higher dimensional generalizations called Anderson $t$-modules play a central role in transcendental number theory over the rational function field $K = \Fq (\theta)$ in one variable over a finite field with $q$ elements. Questions about their algebraic independence are of major interest, and were first systematically investigated by Yu \cite{jy:ttoff, jy:tadm, jy:tadmsv, jy:opaqpodm, jy:taszvicp, jy:ahdm}.

For Drinfeld modules, the above question is very well understood, and the transcendence degree can be computed from the dimension of the endomorphism ring of the Drinfeld module (see \cite{cc-mp:arplr2dm,cc-mp:aipldm}). The proof uses the $t$-comotive\footnote{Traditionally, the comotive was called \emph{dual $t$-motive}, but more recent paper use the term comotive to avoid confusion with the categorical dual of the $t$-motive.} of the Drinfeld module, its motivic Galois group, and its rigid analytic trivialization. In more detail, if the Drinfeld module has rank $r$, the rigid analytic trivialization of its $t$-comotive is an $(r\times r)$-matrix $\Psi$ with coefficients in the Tate algebra $\CI\cs{t}$ over a complete algebraically closed extension $\CI$ of $K$. Papanikolas showed in \cite[Theorem 1.1.7]{mp:tdadmaicl} that the transcendence degree over $K(t)$ of the field $K(t,\Psi)$ generated by the entries of $\Psi$ equals the dimension of the motivic Galois group $\Gamma$, and furthermore that the transcendence degree stays the same when specializing $t$ to $\theta$, i.e.,~
\[ \trdeg(K(\Psi_{|_{t=\theta}}):K)=\trdeg(K(t,\Psi):K(t))= \dim(\Gamma).\]
Furthermore, the $K$-vector space generated by the entries of $\Psi_{|_{t=\theta}}$ is the same as the $K$-vector space generated by the periods and quasi-periods of the Drinfeld module (\cite[Prop.~3.4.7]{cc-mp:aipldm}).
Using this chain of equalities, Chang and Papanikolas had to determine the motivic Galois group $\Gamma$ which turned out to be isomorphic to $\Cent_{\GL_r(K)}(K_{\End})$, the centralizer in $\GL_r(K)$ of the field of fractions $K_{\End}$ of the endomorphism ring of the Drinfeld module. Hence, if the degree of the extension is $s:=[K_{\End}:K]$, then the dimension of the motivic Galois group $\Gamma$ (and hence the transcendence degree of the periods and quasi-periods) is 
\[ \dim(\Gamma)= \frac{r^2}{s}. \]

\medskip

The higher dimensional case of Anderson $t$-modules, however, is much more complex. 
The periods are now tuples, and the rigid analytic trivialization of the corresponding $t$-comotive only gives rise to some of their coordinates. These coordinates are, therefore, called \emph{tractable} coordinates (See \S~\ref{SS:RatPeriodsTractable}).
Furthermore, for Anderson $t$-modules, it is also much harder to determine the motivic Galois groups of the associated $t$-comotive.

The first author \cite{am:ptmsv} by-passed the first problem by using so-called prolongations of the $t$-module/$t$-comotive. The tractable coordinates of these prolongations contain all coordinates of periods of the original $t$-module (see also \cite[Theorem F]{cn-mp:hpqpam}). We will return to these prolongations later.

\medskip

For Anderson $t$-modules defined over a separable algebraic extension of $K$, the quasi-periods and the coordinates of periods are all contained in the separable algebraic closure $K_\infty^\sep$ of $K_\infty=\Fq\ls{\tfrac{1}{\theta}}$ (see e.g.~\cite[p.6, Remarque]{ld:dmdt} for Drinfeld modules, and \cite[Lemma 4.22]{cn-mp:hpqpam} in general), and the fields $K$, $K_\infty$, and $K_\infty^\sep$ have an extra structure: they are equipped with the family of hyperdifferential operators $\hdthe{n}$. \footnote{These hyperdifferential operators are the positive characteristic versions of the  differential operators $\frac{1}{n!}(\partial_\theta)^n$, where $\partial_\theta$ is the usual derivative by $\theta$. (See beginning of Section \ref{S:prolongation-of-group-schemes} for precise definitions.)}

The presence of the hyperdifferential operators 
raised the question of how the hyperderivatives of the periods and quasi-periods fit into the picture, and in particular what the algebraic relations between periods and quasi-periods, and their hyperderivatives, are. 
First partial results for Drinfeld modules were obtained by Brownawell and Denis (see \cite{ld:dmdt}, \cite{ld:iac2}, \cite{db:lidddm-I}, \cite{ld:iadpmc}, \cite{db-ld:lidddm-II}). It was solved for the Carlitz module by the first author \cite{am:aicph}, and for general Drinfeld modules by the second author \cite{cn:arhpldm}. 
However, the question remains open for general Anderson $t$-modules.

\medskip

One aim of our manuscript is to shed more light on this question. We provide a sufficient criterion for Anderson $t$-modules for which the result for Drinfeld modules carries over (see Theorem~\ref{mainthm:transcendence result}).
By applying this theorem to the direct sum of Drinfeld modules, we obtain results on the algebraic relations between periods and quasi-periods, and their hyperderivatives, for several Drinfeld modules together (see Theorem \ref{mainthm:severalDMs}).

The result in \cite{cn:arhpldm} for Drinfeld modules uses a connection between the hyperderivatives of the periods and quasi-periods, and the periods of the prolongation of the Drinfeld module. Such a connection was more recently described for arbitrary Anderson $t$-modules in \cite{cn-mp:hpqpam}, and we use this to trace back the transcendence questions including hyperderivatives to the transcendence questions solely for periods and quasi-periods of prolongations of Anderson $t$-modules.

As mentioned above, Papanikolas' theorem (cf.~\cite[Theorem 1.1.7]{mp:tdadmaicl}) relates these transcendence questions to the dimension of the motivic Galois group.
Therefore, in order to tackle transcendence questions about hyperderivatives of periods and quasi-periods, we need to examine the motivic Galois groups of prolongations of Anderson $t$-modules.

Throughout this paper, we will work with the $t$-motive  and not with the $t$-comotive, as it is better suited to our approach. The transition from the $t$-comotive side to the $t$-motive side, and how Papanikolas' theorem is phrased for $t$-motives, will be explained in Section \ref{Sub:RATsAndGaloisGroups}.

Anderson \cite{ga:tm} proved that an abelian $t$-module is rigid analytically trivial if and only if it is uniformizable. In this paper, we will only consider Anderson $t$-modules that are uniformizable and abelian, as Papanikolas' theorem requires these properties (see Section \ref{SubS:motive} and \ref{Sub:RATsAndGaloisGroups} for the definitions of these three terms).

\subsection{Bounds for motivic Galois groups}\label{Sub:BMGG}
Upper bounds on the motivic Galois groups are usually obtained by symmetries or special structures of the Anderson $t$-motive, whereas lower bounds can be obtained by analyzing the images of associated $\fp$-adic Galois representations for different non-zero prime ideals $\fp$ of $\bA=\Fq[t]$. In fact, the Mumford-Tate conjecture states that for a $t$-module $E$ (\emph{of generic characteristic}) which is defined over a finite extension $L$ of $K$, the $\fp$-adic Galois representation $\varrho_{E,\fp}$ attached to $E$ has a Zariski dense image in the motivic Galois group of $E$ for every prime $0 \ne \fp \in \bA$.

\medskip

For $n \geq 0$, the prolongation $\rho_{n}E$ of a $t$-module $E$ is a successive extension of $E$ with itself, and it turns out that a construction of an $n$-th prolongation of a scheme introduced by Rosen (see Section \ref{S:prolongation-of-group-schemes} for details) is the right object for a general upper bound for the Galois group $\Gamma_{\rho_n\mot}$ of the $t$-motive $\rho_n\mot$ associated to $\rho_{n}E$.

\begin{maintheorem} (see Theorem~\ref{thm:upper-bound-on-group-of-prolongation}) \label{mainthm:upper-bound}
    Let $E$ be an abelian $t$-module and $\mot=\mot_E$ be its corresponding $t$-motive and for $n \geq 0$, $\rho_n \mot$ be its $n$-th prolongation.
    Then the Galois group of the prolongation, $\Gamma_{\rho_n \mot}$, is contained in $\rP{n}(\Gamma_\mot)$, the $n$-th prolongation of the Galois group of $\mot$.
\end{maintheorem}

For investigating the images of the $\fp$-adic Galois representations in the motivic Galois group, we start more generally with linear algebraic groups $G$ over some complete discretely valued field $\cK$, and subgroups $\Delta\subseteq G(\cK)$ of their $\cK$-rational points. Here, the field $\cK$ is equipped with an appropriate family of hyperdifferential operators.

Using a canonical homomorphism $\prol{n}:G(\cK)\to \rP{n}(G)(\cK)$ to the points of the prolongation group $\rP{n}(G)$, we show the following theorem, where $\cO$ denotes the valuation ring of $\cK$ and $\fm$ denotes its maximal ideal.

\begin{maintheorem}(see Theorem \ref{T:prolongated-subgroup-Zariski-dense})\\ \label{mainthm:Zariski-density}
    Let $G\subseteq \GL_{r,\cK}$, and $\Delta\subseteq G(\cO):=G(\cK)\cap \GL_r(\cO)$ be an $\fm$-adically open subgroup.
	Then for each $n\geq 0$, the group $\prol{n}\Delta:=\{ \prol{n}(X)\mid X\in\Delta \}$ is Zariski-dense in $\rP{n}(G)$.
\end{maintheorem}

Applying this theorem and the previous theorem to the setting of Anderson $t$-modules, we obtain a situation where the upper bound in the proposition is sharp, leading to the following theorem. 
Let $L$ be a field such that $K\subseteq L \subseteq \CC_{\infty}$.

\begin{maintheorem} (see Theorem \ref{T:galois-group-of-prolongation})	\\ \label{mainthm:Galois group of prolongation}
Let $(E,\phi)$ be a uniformizable abelian $t$-module over $L$ of rank $r$ 
and let $\mot$ be its associated $t$-motive. Further, let $\Gamma_\mot\subseteq \GL_{r,\bK}$ be the motivic Galois group of $\mot$ which is defined over $\bK=\Fq(t)$, $\Gamma_{\mot,0}$ the connected component of $1$ in $\Gamma_\mot$, and let
	$\fp$ be a prime of $\bA=\Fq[t]$. Further, let $\varrho_\fp:\Gal(L^\mathrm{sep}/L)\to \GL_r(\bA_\fp)$, the $\fp$-adic Galois representation associated to $E$.
	Assume that the $\image(\varrho_\fp)\cap \Gamma_{\mot,0}(\bK_\fp)$ is $\fp$-adically open in $\Gamma_{\mot,0}(\bK_\fp)$. 
	Then for any $n\geq 0$, the Galois group $\Gamma_{\rho_{n}\mot}$ of the $n$-th prolongation $\rho_{n}\mot$ of $\mot$ equals the $n$-th prolongation $\rP{n}(\Gamma_\mot)$ of the Galois group $\Gamma_{\mot}$ of $\mot$.
	In particular, $\dim(\Gamma_{\rho_{n}\mot})=(n+1)\cdot \dim(\Gamma_\mot)$.
\end{maintheorem}

The approach that we pursue also gives more examples where the Mumford-Tate conjecture holds.

\begin{maintheorem} (see Theorem \ref{thm:evidence-Mumford-Tate-conjecture})
\label{mainthm:Mumford-Tate-conjecture}
       Let $(E, \phi)$ be a uniformizable abelian $t$-module over a finite extension $L$ of $K=\Fq(\theta)$ and let $\mot$ be its associated $t$-motive. Assume that for every prime $\fp\subseteq \Fq[t]$, the image of the $\fp$-adic Galois representation $\varrho_{E,\fp}$ is $\fp$-adically open in $\Gamma_\mot(\bK_\fp)$.
    Then for all $n\geq 0$, the Mumford-Tate conjecture holds for the $n$-th prolongation $\rho_n E$ of $E$.
\end{maintheorem}

As mentioned above, the dimension of the motivic Galois group of $n$-th prolongation equals the transcendence degrees of the tractable coordinates of periods and quasi-periods, and their hyperderivatives up to order $n$. 
Hence, from Theorem \ref{mainthm:Galois group of prolongation}, we directly obtain the following theorem.

\begin{maintheorem}\label{mainthm:transcendence result}(see Theorem~\ref{Thm:transcendence-degrees})
Let $(E,\phi)$ be a uniformizable abelian $t$-module over a separable algebraic extension $L$ of $K=\Fq(\theta)$ of rank $r$ and let $\mot$ be its associated $t$-motive. Further, let $\Gamma_\mot\subseteq \GL_{r,\bK}$ be the motivic Galois group of $\mot$ which is defined over $\bK=\Fq(t)$, and let
	$\fp$ be a prime of $\bA=\Fq[t]$. Further, let $\varrho_\fp:\Gal(L^\mathrm{sep}/L)\to \GL_r(\bA_\fp)$ the $\fp$-adic Galois representation associated to $E$.
	Assume that the image of $\varrho_\fp$ is $\fp$-adically open in $\Gamma_\mot(\bK_\fp)$. Let $\mathcal{G}_n$ denote the field generated over $L$ by the quasi-periods and the tractable coordinates of periods of $E$, and their hyperderivatives upto $n$. Then, 
    \begin{equation}
        \trdeg(\mathcal{G}_n:L) = (n+1)\cdot \trdeg(\mathcal{G}_0:L). 
          \end{equation} 
\end{maintheorem}

We have to restrict to separable algebraic extensions $L$ of $K$ here, as we use Papanikolas' theorem (which requires \emph{algebraic}) and hyperderivatives with respect to $\theta$ (which requires \emph{separable}).
We mention that if $n\geq 0$ is chosen so that $(\rd\phi(t)-\theta \Id_d)^n=0$, then all coordinates of periods are contained in $\mathcal{G}_n$ (see Theorem~\ref{T:hyperRATnPer}).  
Thus, we obtain transcendence degrees of all coordinates of periods (see Theorem \ref{Thm:periodentries}).

\subsection{Periods of several Drinfeld modules}

When $(E, \phi)$ is a Drinfeld module, there are no non-tractable coordinates and the transcendence question for its periods and quasi-periods, and their hyperderivatives, was solved by the second author in \cite{cn:arhpldm}.

When considering several Drinfeld modules, there are only a few results about the algebraic relations among their periods and quasi-periods, and no results including their hyperderivatives.
Chang \cite{cc:svdmfai} considered the case of several Drinfeld modules of rank $r=2$ which are of CM type, i.e., where the field of fractions of the endomorphism rings have degree $r$ over $K=\Fq(\theta)$. Very recently, Chen and Gezmiş \cite{yc-og:svmdmfarcmp} extended this result to CM Drinfeld modules of arbitrary ranks.

\begin{Theorem} (\cite[Theorem 1.3.1]{yc-og:svmdmfarcmp})\\
Let $E_1,\ldots, E_m$ be CM Drinfeld modules of ranks $r_1,\ldots, r_m\geq 2$ respectively defined over $\oK$. Suppose that for all $1\leq i\leq m$, the field of fractions $K_i$ of the endomorphism ring of $E_i$ is Galois over $K=\Fq(\theta)$, and that $K_i\cap K_j=\Fq(\theta)$ for any $i \ne j$. Let $\mathcal{F}$ be the field generated by the periods and quasi-periods of all $E_i$. Then we have
\[ \trdeg(\mathcal{F}:K) = 1+\sum_{i=1}^m (r_i-1)= (r_1 + \dots + r_m ) - (m - 1). \]
\end{Theorem} 

We remark that Anderson's Legendre relation (see \cite[\S8.1]{og-mp:tdrifdmota})
implies that for each Drinfeld module, (a finite extension of) the field generated by its periods and quasi-periods always contains the Carlitz period. Hence, the transcendence degree given in the theorem is the maximal possible one under the given conditions.

\medskip

Our method allows us to answer such questions without the restriction to CM-types, but with a hypothesis on the $\fp$-adic Galois representation for some prime $\fp$ of $\bA$.
Furthermore, it includes the hyperderivatives.

\begin{maintheorem} (see Corollary \ref{C:trdegSeveralDM})\\ \label{mainthm:severalDMs}
    Let $C$ be the Carlitz module, and let $E_1, \ldots, E_m$ be Drinfeld modules over some finite separable extension $L$ of $K$. Assume that there exists a prime $\fp \in \bA$ such that the $\fp$-power torsion extensions $L(E_1[\fp^\infty]),\dots, L(E_m[\fp^\infty])$ are linearly disjoint field extensions of $L(C[\fp^\infty])$. 
    Then the only algebraic relations among all the periods and quasi-periods, and their hyperderivatives, are the algebraic relations among those of each $E_i$, $i=1,\dots, m$, Anderson's Legendre relation, and the relations derived thereof by hyperdifferentiating.
    
    In particular for any $n\geq 0$, the field extension $\mathcal{F}_n$ generated over $L$ by all periods and quasi-periods, and their hyperderivatives up to order $n$, satisfies
    \[ \trdeg(\mathcal{F}_n:L) = (n+1)\cdot \left( 1 + \sum_{i=1}^m \Bigl(\frac{r_i^2}{s_i} -1 \Bigr) \right) \]
    where for each $i=1,\ldots, m$, $r_i$ is the rank of $E_i$, and $s_i$ is the degree of the field of fraction of the endomorphism ring of $E_i$ over $\Fq(\theta)$.
\end{maintheorem}

\begin{Remark}
    In the articles \cite{cc-mp:aipldm} and \cite{cn:arhpldm}, the transcendence results are further extended to include logarithms of Drinfeld modules, and their hyperderivatives.
    For doing so in our case, we would need a statement on an explicit embedding of the Galois representations of any $t$-motive into the Galois group of the $t$-motive parallel to \cite[Thm.~E]{am:naamcigp}.
    We therefore decided to postpone the question on logarithms to future work. 
\end{Remark}

\subsection{Outline of the paper}

In Section \ref{S:prolongation-of-group-schemes}, we review Rosen's definition of prolongations in \cite{er:pda} in the case of affine schemes. Then we apply this to affine group schemes and provide some properties of these prolongations of group schemes.
In Section \ref{S:Zdensity}, we consider \emph{prolongated subgroups}, i.e., subgroups of points of such prolongations that are derived from subgroups of the points of the original group scheme via canonical maps $\prol{k}:G(\cK)\to \rP{k}(G)(\cK)$. The central part is 
Theorem \ref{mainthm:Zariski-density}.

Anderson $t$-modules will only appear from Section \ref{S:t-modules-and-motivic-galois-groups} on. This section gathers notation and known results on Anderson $t$-modules, and their $t$-motives to the extent that is relevant for our purpose. Readers that are familiar with these aspects of the theory can skip this section.

In Section \ref{S:Galois groups of prolongations}, we prove our main theorems
on the motivic Galois group of prolongations of $t$-modules (Theorems \ref{mainthm:upper-bound} and \ref{mainthm:Galois group of prolongation}), on the Mumford-Tate conjecture (Theorem \ref{mainthm:Mumford-Tate-conjecture}), and on the transcendence result for periods and quasi-periods (Theorem \ref{mainthm:transcendence result}). We also demonstrate the latter through examples at the end of this section.

In the final Section \ref{S:non-isogenous-Drinfeld-modules}, we illustrate the power of the results in Section~\ref{S:Galois groups of prolongations} by deducing our transcendence result on the periods, quasi-periods, and their hyperderivatives of several Drinfeld modules (Theorem \ref{mainthm:severalDMs}).
Then, we end this paper by computing explicit families of pairs of Drinfeld modules of rank $2$ that satisfy this hypothesis. The method includes a criterion for Drinfeld modules of rank $2$ such that its $t$-adic Galois representation is surjective (cf.~Corollary \ref{cor:surjective-Galois-representation}).

\section{Prolongation of group schemes}\label{S:prolongation-of-group-schemes}

When examining the motivic Galois groups of prolongations of $t$-motives, one encounters a general upper bound for these Galois groups from the special structure of these prolongations.

E.~Rosen \cite{er:pda} defined prolongations of schemes in positive characteristic which suits our Galois groups perfectly, and we will adapt Rosen's definition. We will first review the definition and properties that we will need, and then we will show that the prolongation of an affine group scheme is also an affine group scheme.

We emphasize that in this paper, we consider two different types of \emph{prolongations}: 
\begin{itemize} 
\item the $n$-th prolongations of schemes and group schemes which we denote using $\rP{n}(\cdot)$, and 
\item the $n$-th prolongations of $t$-modules $E$ and $t$-motives $\mot$ which we denote using $\rho_nE$ and $\rho_n\mot$, respectively. 
\end{itemize}
The former is the subject of the current section while the latter will be considered in Section \ref{S:t-modules-and-motivic-galois-groups}. 

\medskip

    Throughout this section, let $\cK$ be a field of characteristic $p>0$. 

\subsection{Higher derivations}
    
    \begin{Definition}\label{D:hyperder}
    A \emph{higher derivation} (or \emph{family of hyperdifferential operators} or \emph{hyperderivatives}) \emph{on} $\cK$ is a family $(\hde{k})_{k\geq 0}$ of additive maps $\hde{k}:\cK\to \cK$ satisfying the rules
    \begin{itemize}
        \item[(a)] $\hde{0}=\id_{\cK}$, and
        \item[(b)] for all $k\geq 0$, $a,b\in \cK$: $\quad \hd{k}{ab}= \sum_{i=0}^k \hd{i}{a}\hd{k-i}{b}$.
    \end{itemize}
    \end{Definition}

    We will consider $(\cK, \hde{*})$ (or $(\cK, (\hde{k})_{k \geq 0{}})$) throughout and write $\cK$ for short. 

With a given higher derivation, we associate the \emph{(generic) Taylor series homomorphism}
\begin{equation}\label{E:LmathcalD}
    \D:\cK\to \cK\ps{T}, a\mapsto \sum_{k=0}^\infty \hd{k}{a}T^k 
\end{equation}
which is a ring homomorphism due to condition (b).

\begin{Example}\label{Example:Hyperd}
For a field $F$ of characteristic $p$, 
an independent variable $x$, and the polynomial ring $F[x]$ over $F$, there is the higher derivation $(\hde{k}_x)_{k\geq 0}$ defined on polynomials by
\[
\hde{k}_{x} \left( \sum_{i} a_i x^i \right) = \sum_{i} \binom{i}{k} a_i x^{i-k}, 
\]
 where $\binom{i}{k}$ denotes the usual binomial coefficient modulo the characteristic $p$. Using  
 a quotient rule, these maps are extended to all of the field of fraction $\cK=F(x)$ or the Laurent series ring $\cK=F(\!(x)\!)$.
Further, for a place $\nu$ of $\cK$, there are unique continuous extensions of $\hde{k}_{x}$ to $\cK_{\nu}^{\sep}$, the separable closure of the completion of $\cK$ at $\nu$ (see e.g.~\cite[Theorems 5 \& 6]{kc:dp}).
\end{Example}

\begin{Definition}
For $\cK$-algebras $R$ and $B$, a \emph{higher derivation of order $n\geq 0$ from $R$ to $B$ over $\cK$} is a $\cK$-algebra homomorphism $\delta_0: R \rightarrow B$ and for $1\leq i \leq n$, abelian group homomorphisms $\delta_i: R \rightarrow B$ such that
\begin{enumerate}
    \item[(a)] $\delta_i(a) = \hde{i}(a)$ for all $a \in \cK, 0\leq i \leq n$,
    \item[(b)] $\delta_i(xy) = \sum_{j=0}^i \delta_j(x) \delta_{i-j}(y)$ for all $x,y\in R$, $0\leq i\leq n$.
\end{enumerate}
Let $\Hde^n_\cK(R, B)$ denote the set of \emph{higher derivations of order $n$} from $R$ to $B$ over $\cK$. 
\end{Definition}

 For a $\cK$-algebra $B$, we denote by $\tilde{B}_n$, the ring $B[T]/(T^{n+1})$ seen as a $\cK$-algebra via the truncated Taylor series homomorphism
	\begin{align} \label{eq:defintion-of-cal-D}
	\cK&\xrightarrow{\D} \cK\ps{T} 
    \to \cK\ps{T}/(T^{n+1})\to B[T]/(T^{n+1})=\tilde{B}_n,\\
    a&\longmapsto \hspace{4.5cm}\sum_{k=0}^n \hd{k}{a}T^k \notag
	\end{align}
    where the last arrow is just the homomorphism given by applying the structure map $\cK\to B$ coefficient-wise.

\begin{Remark}[{see \cite[Rem.~1.5]{er:pda}}]
Let $h:\tilde{B}_n \rightarrow B$ be the homomorphism $g(T) \mapsto g(0)$. Then $\delta_i: B \rightarrow B$ is a higher derivation if and only if the map $D: B \rightarrow \tilde{B}_n$ such that $D(r)=  \delta_0(r) + \delta_1(r)T + \dots + \delta_n(r)T^n$ is a homomorphism with $h \circ D = \Id_B$.
    \end{Remark}

 \begin{Proposition}\cite[Prop.~1.19]{er:pda}
 \label{prop:er-prop.1.19} 
     Let $R, B$ be $\cK$-algebras. Given $(\delta_i)_{0\leq i \leq n} \in \Hde_\cK^n(R,B)$, define a map $\varphi_{\delta}: R \rightarrow \tilde{B}_n$ by $\varphi_\delta(r) = \delta_0(r)+\delta_1(r)T + \dots + \delta_n(r)T^n$. Then, $\varphi_\delta$ is an element of $\Hom_\cK(R, \tilde{B}_n)$ and we have a bijection
     \begin{align*}
         \Hde_\cK^n(R,B&) \rightarrow \Hom_\cK(R, \tilde{B}_n) \\
        & \delta \mapsto \varphi_\delta.
     \end{align*}
 \end{Proposition}

\subsection{Definition and properties of prolongations}

\begin{Definition}[{\cite[Def.~1.9]{er:pda}}]
	Let $R$ be a $\cK$-algebra. The space of \emph{Hasse-Schmidt differentials} of $R$ over $\cK$ of length $n \geq 0$, denoted by 
	$HS^n_{R/(\cK,\hde{*})}$, is defined to be the $R$-algebra that is the quotient of the  polynomial $R$-algebra $R[d_i(x) | x\in R, 1\leq i\leq n]$ (generated by symbols $d_i(x)$) by the ideal $\mathcal{I}$ generated by:
	\begin{enumerate}
		\item[(a)] $d_i(x+y)-d_i(x)-d_i(y)$ for all $x,y\in R$, $1\leq i\leq n$,
		\item[(b)] $d_i(xy) =\sum_{j=0}^i d_j(x) d_{i-j}(y)$ for all $x,y\in R$, $1\leq i\leq n$,
		\item[(c)] $d_i(r)-\hd{i}{r}$ for all $r\in \cK$, $1\leq i\leq n$,
	\end{enumerate}
	where we set $d_0(x):=x$ for all $x\in R$. The space of Hasse-Schmidt differentials $HS^n_{R/(\cK,\hde{*})}$ comes with a universal derivation
	$\hdiffe{*}=(\hdiffe{i}:R\to HS^n_{R/(\cK,\hde{*})})_{0\leq i\leq n}$ such that for each $x\in R$,
	\[   \hdiff{i}{x}=d_i(x) \mod \mathcal{I}. \]
 \end{Definition}

    As in the classical case, this universal derivation satisfies a universal property among all higher derivations from $R$ to $\cK$-algebras.
    This is formalized as follows.

     \begin{Proposition}\cite[Prop.~1.18]{er:pda} \label{prop:er-prop.1.18}
     Let $R, B$ be $\cK$-algebras. Given $(\delta_i)_{0\leq i \leq n} \in \Hde_\cK^n(R,B)$, there exists a unique $\cK$-algebra homomorphism $ \phi: HS^n_{R/(\cK,\hde{*})} \rightarrow B$  such that $\delta_i = \phi \circ \hdiffe{i}$ for all $0\leq i\leq n$, i.e., the map $\Hom_\cK(HS^n_{R/(\cK,\hde{*})}, B) \rightarrow \Hde_\cK^n(R,B)$ is bijective. 
 \end{Proposition}

 \begin{proof}[Sketch of proof]
 The map $\bar{\phi}: R[d_i(x) | x\in R, 1\leq i\leq n] \rightarrow B$ defined by $d_i(x) \mapsto \delta_i(x)$ induces a the map $\phi$ since $\mathcal{I} \subseteq \Ker(\bar{\phi})$. Note that $\phi$ is unique since $\delta_i = \phi \circ \hdiffe{i}$ is unique. 
 \end{proof}
    
 \begin{Remark} (see \cite[Cor.~1.24]{er:pda})\\ \label{rem:explicit-description-of-HS}
	Suppose $R$ is given by generators $x_i$, $i\in I$, subject to relations $f_j$, $j\in J$ for some index sets $I$ and $J$, that is, $R= \cK[x_i: i\in I]/\left(f_j: j\in J\right)$. Then  
	\[ 
	HS^n_{R/(\cK,\hde{*})}\cong R\left[\hdiff{k}{x_i}:i \in I, 1\leq k\leq n\right]/\left(\hdiff{k}{f_j}: j\in J, 1\leq k\leq n\right).
	\]
\end{Remark}	

 Combining Proposition \ref{prop:er-prop.1.19} and Proposition \ref{prop:er-prop.1.18}, we get the following corollary.
\begin{Corollary} \label{cor:rosens-equality} \cite[Cor.1.20]{er:pda} 
	Let $R, B$ be $\cK$-algebras, then there is a natural bijection
	\[ \Hom_\cK\left(HS^n_{R/(\cK,\hde{*})}, B\right) \to \Hom_{\cK}(R,\tilde{B}_n). \]
	
\end{Corollary}

\begin{Definition}\cite[Def.2.3]{er:pda}\label{D:proAGS}
	Let $X=\Spec(R)$ be an affine scheme over $\cK$.
	The \emph{$n$-th prolongation} of $X$ is the affine $\cK$-scheme 
	\[ \rP{n}(X):= \Spec \left(HS^n_{R/(\cK,\hde{*})}\right).\footnote{Rosen uses a more general version \textbf{Spec} for prolongations. However, as our scheme $X$ is affine, \textbf{Spec}$\left(HS^n_{R/(\cK,\hde{*})}\right)$ is also affine and turns out to be the same as $\Spec \left(HS^n_{R/(\cK,\hde{*})}\right)$.} \]
\end{Definition}		

    Translating Corollary \ref{cor:rosens-equality} into the language of schemes, we directly get the following.
    
\begin{Proposition} \label{prop:rosens-equality}
    Let $X=\Spec(R)$ be an affine scheme over $\cK$. Then for all $\cK$-algebras $B$,
	\[  \rP{n}(X)(B) \cong X(\tilde{B}_n). \]
\end{Proposition}

\begin{Lemma}\label{lem:prolongations-preserve-fibre-products}
    The mapping $\rP{n}$ sending affine $\cK$-schemes to affine $\cK$-schemes is a covariant functor that preserves fibre products.
\end{Lemma}

\begin{proof}
    Functoriality is just the schematic version of the functoriality of $HS^n$ which is already given in \cite{er:pda}, so we are left to prove that the functor preserves fibre products, i.e., given affine $\cK$-schemes $X,Y,Z$, and morphisms  $f:X\to Z$ and $g:Y\to Z$,
    we have
    \[   \rP{n}(X\times_Z Y) \cong \rP{n}(X)\times_{\rP{n}(Z)}\rP{n}(Y). \]
    Here, the left hand side is the prolongation of the fibre product of $X$ and $Y$ with respect to the two morphisms $f$ and $g$, and the right hand side is the fibre product with respect to the two morphisms $\rP{n}(f):\rP{n}(X)\to \rP{n}(Z)$ and $\rP{n}(g):\rP{n}(Y)\to\rP{n}(Z)$.
    Letting $X=\Spec(R)$, $Y=\Spec(S)$, $Z=\Spec(T)$, and $f^\#:T\to R$, $g^\#:T\to S$ be the $\cK$-algebra homomorphisms corresponding to the maps $f$ and $g$, it is equivalent to showing that 
    \[  HS^n_{R\otimes_T S/(\cK,\hde{*})} \cong HS^n_{R/(\cK,\hde{*})}\otimes_{HS^n_{T/(\cK,\hde{*})}}HS^n_{S/(\cK,\hde{*})}.\]
    On the one hand, by functoriality, we have a commutative diagram
    
    \[ \xymatrix{ & HS^n_{R/(\cK,\hde{*})} \ar[dr] &\\
    HS^n_{T/(\cK,\hde{*})} \ar[ur] \ar[dr] & & HS^n_{R\otimes_T S/(\cK,\hde{*})}\\
     &  HS^n_{S/(\cK,\hde{*})} \ar[ur] &    
    }  \]          
    By the universal property of the tensor product, this induces a homomorphism
    \[ HS^n_{R/(\cK,\hde{*})}\otimes_{HS^n_{T/(\cK,\hde{*})}}HS^n_{S/(\cK,\hde{*})} \to HS^n_{R\otimes_T S/(\cK,\hde{*})}. \]
    On the other hand, for all $0\leq k\leq n$, we have commutative diagrams 
    \[ \xymatrix{ & R \ar[r]^(.3){\hdiffe{k}_R} & HS^n_{R/(\cK,\hde{*})} \ar[dr] &\\
    T  \ar[ur] \ar[dr]& & & HS^n_{R/(\cK,\hde{*})}\otimes_{HS^n_{T/(\cK,\hde{*})}}HS^n_{S/(\cK,\hde{*})}\\
     & S \ar[r]^(.3){\hdiffe{k}_S} & HS^n_{S/(\cK,\hde{*})} \ar[ur] &    
    }  \]     
    which induce maps
    \[  \delta_k: R\otimes_T S \to HS^n_{R/(\cK,\hde{*})}\otimes_{HS^n_{T/(\cK,\hde{*})}}HS^n_{S/(\cK,\hde{*})},\]
    and it is straightforward to see that the collection $(\delta_k)_{0\leq k\leq n}$ is a higher derivation. Hence, by the universal property of $HS^n_{R\otimes_T S/(\cK,\hde{*})}$, we have an induced homomorphism
    \[ HS^n_{R\otimes_T S/(\cK,\hde{*})}\to HS^n_{R/(\cK,\hde{*})}\otimes_{HS^n_{T/(\cK,\hde{*})}}HS^n_{S/(\cK,\hde{*})} \,.\]
    Using again universal properties, it is not hard to verify that the two constructed homomorphisms are inverse to each other.
\end{proof}

\subsection{Prolongations of group schemes}

\begin{Theorem}	
		Let $G$ be an affine group scheme over $\cK$, then its $n$-th prolongation $\rP{n}(G)$ is an affine group scheme.
\end{Theorem}

\begin{proof}
	By definition, $\rP{n}(G)$ is an affine scheme. By Proposition~\ref{prop:rosens-equality}, it represents a functor to groups, as for any $\cK$-algebra $B$, the points $\rP{n}(G)(B)\cong G(\tilde{B}_n)$ build a group in a natural way. Hence, $\rP{n}(G)$ is indeed a group scheme.
\end{proof}

\begin{Proposition}\label{prop:ses-and-dim-of-prolongations}
	Let $n\geq 1$, and $G$ be an affine group scheme over $\cK$. For any $\cK$-algebra $B$, there is a natural exact sequence 
	\begin{equation}\label{eq:ses-of-prolongations}
	 0 \to \Lie(G)(B)\to 
\rP{n}(G)(B) \xrightarrow{\pi_{n,n-1}} 
\rP{n-1}(G)(B) \to 0.
\end{equation}	 
	If $G$ is smooth, the dimension of the prolongation group scheme $\rP{n}(G)$ is given as
	\[ \dim(\rP{n}(G))=(n+1)\cdot \dim(G).\]
\end{Proposition}
	
\begin{proof}
	The natural projection $\tilde{B}_n=B[T]/(T^{n+1}) \to B[T]/(T^{n})=\tilde{B}_{n-1}$ induces the map
	$\pi_{n,n-1}:\rP{n}(G)(B) \to 
\rP{n-1}(G)(B)$, and its kernel equals
	\begin{eqnarray*}
		\Ker\left(G(B[T^{n}]/T^{n+1})\to G(B)\right)&=&\Ker\left(G(B[\varepsilon]/\varepsilon^2)\to G(B)\right)\\
		&=& \Lie(G)(B),
	\end{eqnarray*}
	by definition of the Lie algebra as the tangent space at $\one$. This proves the exactness of the given sequence.
		
	As $\rP{0}(G)=G$ and $\dim (\Lie(G))=\dim (G)$ for smooth group schemes $G$, the statement on the dimension is inductively obtained using the short exact sequence.
\end{proof}

\begin{Lemma}\label{lem:connected-component}
Let $G$ be an affine group scheme over $\cK$, and $\rP{n}(G)$ its $n$-th prolongation.
Denote by $G_0$ and $\rP{n}(G)_0$ the connected component of $1$ in these two groups respectively.
Then \[ \rP{n}(G_0)=\rP{n}(G)_0.\]
Further, $\rP{n}(G)/\rP{n}(G)_0\cong G/G_0$.
\end{Lemma}

\begin{proof}
This follows directly from Proposition \ref{prop:ses-and-dim-of-prolongations}, as $\Lie(G)$ is connected.
\end{proof}

We will make use of the following group homomorphism.

\begin{Definition}\label{def:gammatohyper}
For an affine group scheme $G$ over $\cK$ and $n\geq 0$, we denote by
\[ \prol{n}: G(\cK)\to \rP{n}(G)(\cK) \]
the natural group homomorphism induced by $\D:\cK\to \cK[T]/T^{n+1}$ using the identification $G(\cK[T]/T^{n+1})\cong \rP{n}(G)(\cK)$. If $\Delta\subseteq G(\cK)$ is a subgroup, we will call its image $\prol{n}(\Delta)\subseteq \rP{n}(G)(\cK)$ a \emph{prolongated} subgroup.

\end{Definition}

In the explicit description of $\rP{n}(G)\subseteq \GL_{r(n+1)}$ for $G\subseteq \GL_r$ given in the next section, this homomorphism is:
\begin{equation}\label{eq:gammatohyper}
	\prol{n}: G(\cK)\to \rP{n}(G)(\cK), X_0\mapsto \prol{n}(X_0)=\begin{pmatrix}
			X_0 & \hd{1}{X_0} & \dots& \hd{n}{X_0} \\
			&    \ddots     & \ddots&\vdots\\
			&&\ddots&\hd{1}{X_0}\\
			&&&X_0
		\end{pmatrix}
\end{equation}
where as usual, for the matrix $X_0$, the term $\hd{i}{X_0}$ is meant to apply $\hde{i}$ entry-wise.

\begin{Remark}
If $B$ is a $\cK$-algebra equipped with a higher derivation $\hde{*}$ extending the one on $\cK$, we have in the same way a group homomorphism 
\begin{equation}\label{eq:gammatohyperoverB}
	\prol{n}: G(B)\to \rP{n}(G)(B), X_0\mapsto \prol{n}(X_0)
\end{equation}
\end{Remark}

\subsection{Explicit description as subgroup of a general linear group}\label{SubS:prolongation-as-subgroup-of-GL}

Consider now an affine group scheme $G$ over $\cK$ with a fixed embedding into some $\GL_{r}$. This means we let $G\subseteq \GL_r$ be defined by an ideal $I_G\subseteq \cK[Y_{ij},1/\det(Y)] $, i.e., for all $\cK$-algebras $B$, and $X\in \GL_r(B)$ we have
\[ X\in G(B) \Longleftrightarrow \text{for all } Q\in I_G, \text{ we have } Q(X)=0. \]
From the equality in Proposition \ref{prop:rosens-equality}, we have
\begin{equation}\label{E:proT}
\rP{m}(G)(B)= G(B[T]/T^{m+1}) 
\end{equation}
with the $\cK$-algebra structure on $\tilde{B}_m=B[T]/T^{m+1}$ given by Formula \eqref{eq:defintion-of-cal-D}, i.e., $\cK \to B[T]/T^{m+1}$ is given by $a \mapsto \sum_{\ell=0}^m\partial^{(\ell)}(a)T^\ell\in \cK[T]/T^{m+1}\subseteq B[T]/T^{m+1}$ (cf.~Equation \eqref{eq:defintion-of-cal-D}). 

Now, for $Q \in I_G$, let $Q^{\D}$ denote the polynomial obtained by applying $\D$ from \eqref{E:LmathcalD} to the coefficients of $Q$. By \eqref{E:proT}, we see that $\rP{m}(G)(B)$ consists of those matrices of the form $X_0+X_1T+\ldots+X_mT^m$ ($X_j\in \Mat_{r\times r}(B)$) such that
\[ Q^\D(X_0+X_1T+\ldots+X_mT^m) \equiv 0 \mod T^{m+1}\, \text{ for all } Q\in I_G.
\]
Grouping together coefficients of each  $T^\ell$, $0 \leq \ell \leq m$, suppose
\begin{equation}\label{E:QS} 
Q^\D(Y_0+ Y_1T+ \dots+ Y_m T^m) \equiv  S_0 + S_1 T + \dots + S_m T^m \mod T^{m+1},
\end{equation}
where $S_\ell \in \cK[Y_0, Y_1, \dots, Y_m]$. 

Note that we have a natural $\cK$-algebra embedding $\cK[T]/(T^{m+1})\to \Mat_{(m+1)\times (m+1)}(\cK)$ by mapping $T$ to the nilpotent matrix 
		\[ N=  \begin{pmatrix}
 0 & 1 & 0 & \cdots & 0 \\ \vdots & \ddots & \ddots & \ddots & \vdots  \\ \vdots && \ddots & \ddots & 0 \\ \vdots && & \ddots & 1 \\
 0 &  \cdots &\cdots &\cdots & 0
\end{pmatrix}. \]
Thus, the $\cK$-algebra embedding above allows us to identify $\rP{m}(G)(B)$ with the subgroup of $\GL_{r(m+1)}(B)$ of matrices
	of the  block form
	\begin{equation}
    \begin{pmatrix}
			X_0 & X_1 & \dots&X_m \\
			&    \ddots     & \ddots&\vdots\\
			&&\ddots&X_1\\
			&&&X_0
		\end{pmatrix},
	\label{eq:block-matrix}
	\end{equation} 
	where the ideal of relations $I_{\rP{m}(G)}\subseteq \cK[Y_0,\ldots, Y_m,1/\det(Y_0)]$ is generated by all $S_\ell$, $0\leq \ell \leq k$ as in \eqref{E:QS} for each $Q \in I_G$. 
    Of course, $\cK[Y_0,\ldots, Y_m,1/\det(Y_0)]/I_{\rP{m}(G)}$ is exactly the description of $HS^n_{\cK[Y,1/\det(Y)]/(\cK,\hde{*})}$ given in Remark \ref{rem:explicit-description-of-HS}.

\begin{Example}
Suppose $r=2$, $m=1$, and $I_G \subseteq \cK[Y_{ij}, 1/\det(Y)]$ is defined by the polynomials 
\[
Q_1=Y_{11}^2+aY_{21}^2 -1, \, Q_2=Y_{11}Y_{12}+aY_{21}Y_{22}, \text{ and } Q_3=Y_{12}^2+aY_{22}^2 -a,
\]
for some fixed $a\in \cK$, giving the condition
\[
X^{\tr}\begin{pmatrix}
1 & 0\\
0 & a
\end{pmatrix} X - \begin{pmatrix}
1 & 0\\
0 & a
\end{pmatrix}=0 \text{ for all } X\in G(B).
\]
Then, 
\begin{align*}
Q_1^\D(Y_0+ Y_1T) &= \bigl((Y_0)_{11} + (Y_1)_{11}T\bigr)^{2} + \bigl(a + \hd{1}{a}T\bigr)\bigl((Y_0)_{21} + (Y_1)_{21}T\bigr)^2 -1 \\
&\equiv \Bigl((Y_0)_{11}^{\,2}+a(Y_0)_{21}^{\,2}-1\Bigr) \\
& \qquad+ \Bigl(2(Y_0)_{11}(Y_1)_{11}+2a(Y_0)_{21}(Y_1)_{21} +\hd{1}{a}(Y_0)_{21}^{\,2}\Bigr)T \mod{T^2},\\
&\\
Q_2^\D(Y_0+ Y_1T) &= \bigl((Y_0)_{11}+(Y_1)_{11}T\bigr)\bigl((Y_0)_{12}+(Y_1)_{12}T\bigr) \\
&\qquad +\bigl(a +\hd{1}{a}T\bigr)\bigl((Y_0)_{21}+(Y_1)_{21}T\bigr)\bigl((Y_0)_{22}+(Y_1)_{22}T\bigr)\\
&\equiv \Bigl((Y_0)_{11}(Y_0)_{12} + a(Y_0)_{21}(Y_0)_{22}\Bigr) \\
&\qquad + \Bigl((Y_0)_{11}(Y_1)_{12}+(Y_1)_{11}(Y_0)_{12}+a(Y_1)_{21}(Y_0)_{22}\\ &\hspace{1.3cm} + a(Y_0)_{21}(Y_1)_{22} + \hd{1}{a}(Y_0)_{21}(Y_0)_{22}\Bigr) T
\mod{T^2},\\
& \\
Q_3^\D(Y_0+ Y_1T) &= \bigl((Y_0)_{12} + (Y_1)_{12}T\bigr)^{2} + \bigl(a + \hd{1}{a}T\bigr)\bigl((Y_0)_{22} + (Y_1)_{22}T\bigr)^2 - \bigl(a + \hd{1}{a}T\bigr) \\
&\equiv \Bigl((Y_0)_{12}^{\,2}+a(Y_0)_{22}^{\,2}-a\Bigr)\\
&\qquad + \Bigl(2(Y_0)_{12}(Y_1)_{12}+2a(Y_0)_{22}(Y_1)_{22}
+\hd{1}{a}(Y_0)_{22}^{\,2}- \hd{1}{a}\Bigr)T \mod{T^2},
\end{align*}
Then, $I_{\rP{1}(G)} \subseteq \cK[Y_0, Y_1, 1/\det(Y_0)]$ is generated by 
\begin{align*}
& (Y_0)_{11}^{\,2}+a(Y_0)_{21}^{\,2}-1, \\
&(Y_0)_{11}(Y_0)_{12} + a(Y_0)_{21}(Y_0)_{22},\\
&(Y_0)_{12}^{\,2}+a(Y_0)_{22}^{\,2}-a, \\
 & 2(Y_0)_{11}(Y_1)_{11}+2a(Y_0)_{21}(Y_1)_{21} +\hd{1}{a}(Y_0)_{21}^{\,2}, \\
&(Y_0)_{11}(Y_1)_{12}+(Y_1)_{11}(Y_0)_{12}+a(Y_1)_{21}(Y_0)_{22}+ a(Y_0)_{21}(Y_1)_{22} + \hd{1}{a}(Y_0)_{21}(Y_0)_{22}, \text{ and } \\
&2(Y_0)_{12}(Y_1)_{12}+2a(Y_0)_{22}(Y_1)_{22} +\hd{1}{a}(Y_0)_{22}^{\,2}- \hd{1}{a}.
\end{align*}

\end{Example}

\section{Zariski density of prolongated subgroups}\label{S:Zdensity}

In this section, we study prolongated subgroups $\prol{m}(\Delta)$ defined in Definition \ref{def:gammatohyper}. The aim is to give a sufficient criterion on $\Delta\subseteq G(\cK)$ to ensure that its prolongated subgroup is Zariski dense in the prolongation $\rP{m}(G)$.

For stating the criterion, we need $\cK$ to be a complete nonarchimedean field, and we make use of the induced topologies on the groups $G(\cK)$ and $\rP{m}(G)(\cK)$.

The family of hyperdifferential operators on $\cK$ is only needed in Subsection~\ref{SubS:procongsub}. 

For the whole section, let $\cK$ be a complete discrete nonarchimedean local field, $\cO\subset \cK$ its valuation ring, $\fm\subset \cO$ its maximal ideal, as well as $\ck=\cO/\fm$ the corresponding residue field. In addition, we let $\varpi\in \fm$ be a uniformizer, i.e.,~a generator of the ideal $\fm$. As $\cK$ is assumed to have positive characteristic, one has an isomorphism of $\ck$-algebras $\cO\cong \ck\ps{\varpi}$.

\subsection{Congruence subgroups}\label{SubS:congruenceSubgroups}

Before we define congruence subgroups, we state a lemma that will be used in the proof of Lemma~\ref{lem:isomorphism-to-Lie-algebra}.

\begin{Lemma}\label{lem:lifting-of-homomorphisms}
    Let $\cA$ be a finitely generated $\cO$-algebra such that $\cA_\cK:=\cK\otimes_\cO \cA$ is a regular $\cK$-algebra, and let $\cM$ be a relatively maximal ideal of degree $1$ (i.e.,~the kernel of some $\cO$-algebra homomorphism $\cA\to \cO$).
    Then there exists $N_0\geq 0$ such that for all $N\geq N_0$, and any homomorphism $\varphi: \cA\to \cO/\fm^{2N}$ satisfying $\varphi(\cM)\subseteq \fm^N\cO/\fm^{2N}$, there is a homomorphism $\tilde{\varphi}:\cA\to \cO$, such that for all $a\in \cA$, we have $\tilde{\varphi}(a) \equiv \varphi(a)\mod \fm^{N+1} $.    
\end{Lemma}

\begin{proof}
By Noether normalization for $\cA_\cK$, there are algebraically independent elements $\xi_1,\ldots, \xi_d\in \cM\cA_\cK$, and elements $y_1,\ldots, y_s\in \cA_\cK$ that are integral over $\cK[\xi_1,\ldots, \xi_d]$, such that
$\cA_\cK$ is generated as a $\cK$-algebra by $\xi_1,\ldots, \xi_d,y_1,\ldots, y_s$.
Without loss of generality, $y_1,\ldots,y_s\in \cM\cA_\cK$.

By rescaling, we can further assume $\xi_1,\ldots, \xi_d\in \cM\subseteq \cA$, and $y_1,\ldots,y_s$ are in $\cA$, and that they are integral over $\cO[\xi_1,\ldots, \xi_d]$.

As $\cA_\cK$ is regular and $\cA_\cK/(\cM) \cong \cK$, the completion of $\cA_\cK$ with respect to $(\cM)$ is isomorphic to $\cK\ps{\xi_1,\ldots, \xi_d}$, and since $y_i$ is integral over $\cO[\xi_1,\ldots,\xi_d]$, we have $y_i\in \cO\ps{\xi_1,\ldots, \xi_d}$.

We now fix $\cO$-generators $x_1,\ldots, x_n$ for $\cA$ that lie in $\cM$. Then there is some integer $l\geq 0$ such that or all $1\leq k\leq n$, we have $x_k\in \frac{1}{\varpi^l} \cO[\xi_1,\ldots, \xi_d,y_1,\ldots, y_s]$. In particular, this induces an embedding $\cA\subseteq \cK\otimes \cO\ps{\xi_1,\ldots,\xi_d}$.

We claim that $N_0:=l+1$ satisfies the condition in the statement.

Indeed, let $N\geq l+1$, and let $\varphi: \cA\to \cO/\fm^{2N}$ be a homomorphism with $\varphi(\cM)\subseteq \fm^N\cO/\fm^{2N}$.
In particular, $\varphi(\xi_j)\in \fm^N\cO/\fm^{2N}$ for all $j=1,\ldots, d$. Hence, $\varphi$ induces a continuous homomorphism $\hat{\varphi}:\cO\ps{\xi_1,\ldots, \xi_d}\to \cO/\fm^{2N}$ which can be lifted continuously to
$\tilde{\hat{\varphi}}:\cO\ps{\xi_1,\ldots, \xi_d}\to \cO$.
So, for all $j=1,\ldots, d$, $\tilde{\hat{\varphi}}(\xi_j)\equiv \varphi(\xi_j) \mod{\fm^{2N}}$, and for all $i=1,\ldots, s$, 
$\tilde{\hat{\varphi}}(y_i)\equiv \varphi(y_i) \mod{\fm^{2N}}$.

We define $\tilde{\varphi}:\cA\to \cK$ to be the restriction of the scalar extension $\cK\otimes \tilde{\hat{\varphi}}:\cK\otimes \cO\ps{\xi_1,\ldots,\xi_d}\to \cK$.

Since, $\varpi^l x_k\in \cO[\xi_1,\ldots, \xi_d,y_1,\ldots, y_s]\cap \cM$ for all $k=1,\ldots,n$, we obtain
\[\tilde{\varphi}(\varpi^l x_k)\equiv \varphi(\varpi^l x_k) \mod \fm^{2N},
\]
and therefore using $\fm^{N+1}\supseteq \fm^{2N-l}$, we get
\[ \tilde{\varphi}(x_k)\equiv \varphi(x_k) \mod \fm^{N+1}.
\]
As $x_1,\ldots, x_n$ generate $\cA$, this proves the claim.
\end{proof}

For the rest of this section, 
let $G/\cK$ be a smooth affine group scheme with a fixed closed embedding $G\subseteq \GL_r$ into the general linear group.
For our statement, we need congruence subgroups that we define using the embedding in $\GL_r$.

We define the following subgroups of $G(\cK)$:
\[  G(\cO) := G(\cK)\cap \GL_r(\cO)\subseteq \GL_r(\cK), \]
and for any ideal $\fa$ of $\cO$,
\[    G(\fa) := \ker\Bigl( G(\cO)\to \GL_r(\cO/\fa)\Bigr)= \bigl\{ g\in G(\cO) \mid  g\equiv \one_r \mod \fa \bigr\}. \] 

Further, using the induced embedding $\Lie(G)(\cK)\subseteq \Mat_r(\cK)$ obtained from $G\subseteq \GL_r$, we define the subspace
\[  \Lie(G)(\cO) := \Lie(G)(\cK)\cap \Mat_r(\cO)\subseteq \Mat_r(\cK). \] 
As $\Lie(G)(\cK)$ is a vector subspace of $\Mat_r(\cK)$, $\Lie(G)(\cO)$ is a free submodule of $\Mat_r(\cO)$ with
\[  \rk_{\cO}(\Lie(G)(\cO)) = \dim_{\cK}(\Lie(G)(\cK)).\]

\begin{Remark}
    We will later apply this to the case of $\cK=\bK_\fp$, $\cO=\bA_\fp$, where $\fp$ is a non-zero prime ideal of $\bA=\Fq[t]$, and $G$ will be the Galois group of a $t$-motive, respectively, of its rigid analytic trivialization.
\end{Remark} 

\begin{Lemma}\label{lem:isomorphism-to-Lie-algebra}
    For all $i\geq 0$, we have injective homomorphisms
    \[ G(\fm^i)/G(\fm^{i+1}) \longrightarrow \Lie(G)(\cO)\otimes_\cO (\cO/\fm).\]
    Furthermore, there exists $N_0\geq 0$ (depending on $G\subseteq \GL_r$) such that for all $i\geq N_0$, these homomorphisms are isomorphisms.

    In particular, for all $i\geq 0$,
    \[   \dim_{\kappa_\fm}\left(G(\fm^i)/G(\fm^{i+1})\right) \leq  \dim_\cK(\Lie(G))=\dim(G) \]
    with equality if $i\geq N_0$.
\end{Lemma}

\begin{proof}
The homomorphisms are given by sending the residue class of $\one_r+\varpi^i A\in G(\fm^i)$ with $A\in \Mat_r(\cO)$ to the residue class of $A$ modulo $ \fm$. For $G=\GL_r$, this is clearly an isomorphism. For general $G$, it is just the restriction of that isomorphism, but we have to work a bit more to show that it is an isomorphism for large enough $i$.

The closed embedding $G\subseteq \GL_r$ induces a surjective map $\cK[\GL_r]\to \cK[G]$ of the rings of regular functions, and hence an isomorphism $\cK[\GL_r]/I\cong \cK[G]$ for an appropriate ideal $I$ of $\cK[\GL_r]$. 
Let $I^{\res}:=I\cap \cO[\GL_r]$, and $\cO[G]:=\cO[\GL_r]/I^{\res}$, so that $G(\cO)\cong \Hom_\cO(\cO[G],\cO)$. Let $u\in\Hom_\cO(\cO[G], \cO)$ be the map corresponding to the unit in $G(\cO)$, as well as $\cM=\Ker(u)$, and recall that $\cO\cong \ck\ps{\varpi}$.
Furthermore, let $N_0$ be the number given by Lemma \ref{lem:lifting-of-homomorphisms} for $\cA=\cO[G]$, and we claim that this $N_0$ is the desired lower bound for $i$.

Consider the diagram,
\[ \xymatrix{
    G(\fm^i) \ar@{^{(}->}[r] \ar@{-->}[ddd]& \Hom_\cO\left( \cO[G], \cO\right) \ar[r] \ar[d]^{\mod{\fm^{2i}}}&  \Hom_\cO\left( \cO[G], \cO/\fm^i\right) \ar[d]^{\mod{\fm}}\\
  & \Hom_\cO\left( \cO[G], \ck\ps{\varpi}/(\varpi^{2i})\right) \ar[r]
 \ar[d]^{\mod{\fm^{i+1}}}& \Hom_\cO\left( \cO[G], \ck\right) \\
  & \Hom_\cO\left( \cO[G], \ck\ps{\varpi}/(\varpi^{i+1})\right) \ar[r] & \Hom_\cO\left( \cO[G], \ck\right)  \ar@{=}[u] \\
  \Lie(G)(\cO)\otimes_\cO (\cO/\fm) \ar@{^{(}->}[r] & \Hom_\cO\left( \cO[G], \ck[\epsilon]/(\epsilon^{2})\right) \ar[r] \ar[u]_{\epsilon\mapsto \varpi^i} 
 \ar@/^2pc/@<12ex>[uu]^{\epsilon\mapsto \varpi^i} & \Hom_\cO\left( \cO[G], \ck\right) \ar@{=}[u]\\
}.\]
The arrows on the right are all induced by homomorphisms on the target, and the labels indicate the maps.

By definition, $G(\fm^i)$ consists of those homomorphisms in $\Hom_\cO\left( \cO[G], \cO\right)$ which have the same image in $\Hom_\cO\left( \cO[G], \cO/\fm^i\right)$ as the homomorphism $u$.

By definition of the Lie algebra, and since $\Lie(G)(\cO)$ is a free $\cO$-module, $\Lie(G)(\cO)\otimes_\cO (\cO/\fm)$ consists of those homomorphisms $\Hom_\cO\left( \cO[G], \ck[\epsilon]/(\epsilon^{2})\right)$ which have the same image in $\Hom_\cO\left( \cO[G], \ck\right)$ as the homomorphism $u$.

\medskip

Now, if we are given $g\in G(\fm^i)$, its image in $\Hom_\cO\left( \cO[G], \ck\ps{\varpi}/(\varpi^{i+1})\right)$ in the third row is a homomorphism $\varphi$ with $\varphi(\cO[G])\subseteq \ck\ps{\varpi^i}/(\varpi^{i+1})$. In particular, $\varphi$ is in the image of the map given by $\epsilon\to \varpi^i$. Since, $g$ and $u$ have the same image in the right column, we obtain an element in $\Lie(G)(\cO)\otimes_\cO (\cO/\fm)$.
We also see that this element is $0$, i.e., that the corresponding element in $\Hom_\cO\left( \cO[G], \ck[\epsilon]/(\epsilon^{2})\right)$ equals $u$, if and only if $g\in G(\fm^{i+1})$. So the homomorphism
 \[ G(\fm^i)/G(\fm^{i+1}) \longrightarrow \Lie(G)(\cO)\otimes_\cO (\cO/\fm)\]
is injective.

Conversely, given an element $\lambda\in \Lie(G)(\cO)\otimes_\cO (\cO/\fm)$, we obtain a homomorphism $\varphi$ in the second row. Since $\lambda(\cM)\subseteq \epsilon \ck[\epsilon]/(\epsilon^2)$, we have $\varphi(\cM)\subseteq \varpi^i$. For $i\geq N_0$,
Lemma \ref{lem:lifting-of-homomorphisms} ensures that  we can find a homomorphism $\tilde{\varphi}\in \Hom_\cO\left( \cO[G], \cO\right)$ such that $\tilde{\varphi}(x)\equiv \varphi(x)\mod{\varpi^{i+1}}$ for all $x\in \cO[G]$.
So on one hand, this map $\tilde{\varphi}$ has the same image in $\Hom_\cO\left( \cO[G], \cO/\fm^i\right)$ as $u$, i.e., $\tilde{\varphi}\in G(\fm^i)$, and on the other hand, it also corresponds to $\lambda$ under the described map.
\end{proof}

\begin{Lemma}\label{lem:Zariski-dense}
	Assume that $G\subseteq \GL_r$ is connected, and let $N_0\geq 0$ satisfy the condition in Lemma \ref{lem:isomorphism-to-Lie-algebra}.
	Let $\Delta$ be a subgroup of $G(\cO)$ such that for some $j\geq N_0$, we have $\Delta\cdot G(\fm^{j+1})\supset G(\fm^j)$. Then $\Delta\subseteq G(\cK)$ is Zariski-dense in $G$. 
\end{Lemma}

\begin{proof}
Let $H\subseteq G$ be the Zariski-closure of $\Delta$ in $G$, and as above $H(\cO):=H(\cK)\cap \GL_r(\cO)$, $H(\fm^i):=H(\cK)\cap \GL_r(\fm^i)$.
This means that for all $l\geq 0$, we have $H(\fm^l)\cap G(\fm^{l+1})=H(\fm^{l+1})$.

For $j\geq N_0$ such that $\Delta\cdot G(\fm^{j+1})\supset G(\fm^j)$, we obtain
\[ G(\fm^j)=\left(\Delta\cdot G(\fm^{j+1})\right)\cap G(\fm^j)=\left(\Delta\cap G(\fm^j)\right)\cdot G(\fm^{j+1})
\subseteq H(\fm^j)\cdot G(\fm^{j+1}).\]

Therefore,
\[  G(\fm^j)/G(\fm^{j+1})\subseteq \left(H(\fm^j)\cdot G(\fm^{j+1})\right)/ G(\fm^{j+1})
\cong H(\fm^j)/\left(H(\fm^j)\cap G(\fm^{j+1})\right)=H(\fm^j)/H(\fm^{j+1}).\]

In particular using Lemma \ref{lem:isomorphism-to-Lie-algebra},
\[ \dim(G) = \dim_{\ck}\left( G(\fm^j)/G(\fm^{j+1})\right) \leq \dim_{\ck}\left( H(\fm^j)/H(\fm^{j+1})\right) \leq \dim(H).\]
As $H$ is a subgroup of $G$, this implies that the dimensions are equal, and as $G$ is connected, $H=G$.
\end{proof}

\begin{Definition}\label{def:Delta-bar}
Let $G$  be as above, and let $\Delta\subseteq G(\cO)$ be a subgroup.
For any $\ell>j\geq 0$, we set
\[ \overline{\Delta}_{j,\ell}:= \left(\Delta\cap G(\fm^j)\right)/G(\fm^\ell) \subseteq G(\fm^j)/G(\fm^\ell). \]
\end{Definition}

Using this definition, Lemma \ref{lem:Zariski-dense} can then be rephrased as that $\Delta$ is Zariski-dense in $G$, if for some
$j\geq N_0$, we have $\overline{\Delta}_{j,j+1}= G(\fm^j)/G(\fm^{j+1})$.

\subsection{Prolongated congruence subgroups}\label{SubS:procongsub}

In addition to the assumptions on $\cK$ from the beginning of the section, we now assume that $\cK$ is equipped with a higher derivation $(\hde{n})_{n\geq 0}$ which restricts to a higher derivation on $\cO$. 

We start with an easy observation.

\begin{Lemma}\label{lem:congruence-mod-fm}
For all $0\leq n< m$, one has
\[ \hd{n}{\varpi^m} \equiv 0 \mod \fm, \]
and
\[ \hd{m}{\varpi^m} \equiv \hd{1}{\varpi}^m \mod \fm. \] 
\end{Lemma}	
\begin{proof}
This follows by using the definition of hyperdifferentiation  and the product rule $\hd{n}{fg} = \sum_{i=0}^j \hd{i}{f}\hd{j-i}{g}$ in a straightforward manner and therefore, we leave this task to the reader.
\end{proof}

As in Subsection~\ref{SubS:congruenceSubgroups}, we let $G/\cK$ be a linear algebraic group with a fixed embedding $G\subseteq \GL_r$. Since $\cK$ is equipped with a higher derivation, we have the prolongations $\rP{m}(G)$ of $G$, and the group homomorphisms $\prol{m}:~G(\cK)\to \rP{m}(G)(\cK)$ (see Definition \ref{def:gammatohyper}). 

We use the induced embedding $\rP{m}(G)\subseteq \GL_{r(m+1)}$ described in Section \ref{SubS:prolongation-as-subgroup-of-GL}, as well as
$\rP{m}(G)(\cO)=\rP{m}(G)(\cK)\cap \GL_{r(m+1)}(\cO)$ 
and $\rP{m}(G)(\fa)$ defined in the previous Section \ref{SubS:congruenceSubgroups} for arbitrary embedded linear algebraic groups.

\begin{Lemma}\label{lem:restriction-of-prol-map}
    Using $G(\cO)$ and $\rP{m}(G)(\cO)$ as above,
    the homomorphism $\prol{m}:~G(\cK)\to \rP{m}(G)(\cK)$ restricts to a homomorphism
    \[ \prol{m}:~G(\cO)\to \rP{m}(G)(\cO).\]
\end{Lemma}

\begin{proof}
    By Formula \ref{eq:gammatohyper}, the homomorphism $\prol{m}$ is explicitly given by
   \[  \prol{m}(X)=\begin{pmatrix}
			X & \hd{1}{X} & \dots& \hd{m}{X} \\
			&    \ddots     & \ddots&\vdots\\
			&&\ddots&\hd{1}{X}\\
			&&&X
		\end{pmatrix} \]
        for $X\in G(\cK)$. Since, the higher derivation stabilizes $\cO$, for $X\in \GL_r(\cO)$, we have $\hd{k}{X}\in \Mat_r(\cO)$ for all $1\leq k \leq m$. Hence for $X\in G(\cO)=G(\cK)\cap \GL_r(\cO)$, we have $\prol{m}(X)\in \rP{m}(G)(\cK)\cap \GL_{r(m+1)}(\cO)=\rP{m}(G)(\cO)$.
\end{proof}

\begin{Proposition}\label{prop:rho-bar}
Let $\textrm{char}(\cK)=p>0$, let $m\geq 1$, and let $N_0\geq 0$ be the number given in Lemma \ref{lem:isomorphism-to-Lie-algebra} for $G$.
Let $\ell=p^e\cdot \ell'\in \NN$ with $p^e>m$, and $\ell \geq N_0$. 

Further let $\overline{\prol{m}}^\ell$ be the composition of $\prol{m}$ with a projection map
\[ \overline{\prol{m}}^\ell: G(\cO) \xrightarrow{\prol{m}} \rP{m}(G)(\cO) \to \rP{m}(G)(\cO)/\rP{m}(G)(\fm^{\ell+1}),\]
sending $X\in G(\cO)$ to the residue class of $\prol{m}(X)=X+\hd{1}{X}T+\ldots+ \hd{m}{X}T^m$.

Then the image of $G(\fm^\ell)$ under $\overline{\prol{m}}^\ell$ contains $\rP{m}(G)(\fm^\ell)/\rP{m}(G)(\fm^{\ell+1})$.
\end{Proposition}

\begin{proof}
We first show that $G(\fm^{\ell+m})$ is mapped via $\overline{\prol{m}}^\ell$ surjectively onto 
\begin{eqnarray*}
\fK & := & \ker\biggl(\rP{m}(G)(\fm^\ell)/\rP{m}(G)(\fm^{\ell+1})\xrightarrow{\overline{\pi}_{m,m-1}} \rP{{}m-1}(G)(\fm^\ell)/\rP{{}m-1}(G)(\fm^{\ell+1})\biggr)\hspace*{1cm}\\
 &&\hspace*{2cm} \cong
\Lie(G)(\fm^\ell)/\Lie(G)(\fm^{\ell+1})= \fm^\ell \cdot \left( \Lie(G)(\cO)/\Lie(G)(\fm)\right).
\end{eqnarray*}
Here $\overline{\pi}_{m,m-1}$ is the map induced by $\pi_{m,m-1}$ in Prop.~\ref{prop:ses-and-dim-of-prolongations}.

Every $X \in G(\fm^{\ell+m})$ is of the form $X=\mathbf{1}+\varpi^{\ell+m}Y$ with $Y\in \Mat(\cO)$. 

Note that for $1 \leq j \leq m$, we have $\hd{j}{\varpi^\ell}=0$, since $\ell$ is divisible by a $p$-power which is larger than~$j$. Then for $1 \leq r \leq m$,
\begin{eqnarray*}
\hd{r}{X} &=& \hd{r}{\varpi^{\ell+m}Y} = \varpi^{\ell} \cdot \hd{r}{\varpi^mY}\\
&=& \varpi^{\ell} \cdot \left( \sum_{n=0}^r \hd{n}{\varpi^m} \hd{r-n}{Y} \right). 
\end{eqnarray*}
By Lemma \ref{lem:congruence-mod-fm}, for $r<m$ all terms of the sum are divisible by $\varpi$, and for $r=m$,
the only term of the sum not divisible by $\varpi$ is $\hd{1}{\varpi}^mY$. 

Hence, $\hd{r}{X}\equiv 0 \mod \fm^{\ell+1}$ for $1\leq r<m$, and $\hd{m}{X}\equiv \varpi^{\ell} \hd{1}{\varpi}^mY\mod \fm^{\ell+1}$.

Hence, the image of $\overline{\prol{m}}^\ell(X)$ in $\rP{{}m-1}(G)(\fm^\ell)/\rP{{}m-1}(G)(\fm^{\ell+1})$ is $\mathbf{1}$, and
\[\overline{\prol{m}}^\ell(X)= \varpi^{\ell} \hd{1}{\varpi}^mY\in \Lie(G)(\fm^\ell)/\Lie(G)(\fm^{\ell+1}).\]
This also shows that up to a non-zero scalar factor, this map agrees with the projection map
\begin{align*}
G(\fm^{\ell+m})\to G(\fm^{\ell+m})/G(\fm^{\ell+m+1}) &\stackrel{\text{Lem.\ref{lem:isomorphism-to-Lie-algebra}}}{=}\Lie(G)(\fm^{\ell+m})/\Lie(G)(\fm^{\ell+m+1}) \\
& = \fm^{\ell+m} \cdot \Lie(G)(\cO)/\Lie(G)(\fm).
\end{align*} 
Hence, $G(\fm^{\ell+m})$ surjects onto $\fK$.

\medskip

We will now show by induction on $m$ that $\overline{\prol{m}}^\ell\left(G(\fm^\ell)\right)$ contains $\rP{m}(G)(\fm^\ell)/\rP{m}(G)(\fm^{\ell+1})$.

For $m=0$, the statement is trivial, as $\rP{0}(G)=G$, and $\overline{\prol{0}}^\ell$ is just the projection map.

For the induction step, we first recognize that $\overline{\pi}_{m,m-1}\circ \overline{\prol{m}}^\ell = \overline{\prol{m-1}}^\ell$.\\
Let $Z\in \rP{m}(G)(\fm^\ell)/\rP{m}(G)(\fm^{\ell+1})$. Then by induction hypothesis, there exists $X_1\in G(\fm^{\ell})$ with
$\overline{\prol{m-1}}^\ell(X_1)=\overline{\pi}_{m,m-1}(Z)\in \rP{{}m-1}(G)(\fm^\ell)/\rP{{}m-1}(G)(\fm^{\ell+1})$.
Therefore, $\overline{\prol{m}}^\ell(X_1)^{-1} \cdot Z \in \fK$, and by the first part of the proof, there is $X_2\in G(\fm^{\ell+m})$ such that $\overline{\prol{m}}^\ell(X_2)=\overline{\prol{m}}^\ell(X_1)^{-1} \cdot Z$.
Hence, $\overline{\prol{m}}^\ell(X_1X_2)=Z$.
\end{proof}

\begin{Remark}
If we look more closely into the proof above, we see that for the given choices of $m$ and $\ell$, the map $\overline{\prol{m}}^\ell$ induces an isomorphism
\[ \overline{\prol{m}}^\ell: G(\fm^\ell)/G(\fm^{\ell+m+1})\to \rP{m}(G)(\fm^\ell)/\rP{m}(G)(\fm^{\ell+1}).\]

\end{Remark}
	
\begin{Theorem}\label{T:prolongated-subgroup-Zariski-dense}
Let $G\subseteq \GL_r$, and $\Delta\subseteq G(\cO)$ be an $\fm$-adically open subgroup.
	Then for each $m\geq 1$, the group $\prol{m}\Delta:=\{ \prol{m}(X)\mid X\in\Delta \}$ is Zariski-dense in $\rP{m}(G)$.
\end{Theorem}	

\begin{proof}
We first assume that $G$ is connected.

As $\Delta$ is $\fm$-adically open in $G(\cO)$, there is some $\ell\geq 0$ such that $\Delta\supset G(\fm^\ell)$. In particular, we can choose $\ell$ to be larger than $N_0$ of Lemma \ref{lem:isomorphism-to-Lie-algebra} and to be divisible by a $p$-power that is larger than $m$. Applying the map $\overline{\prol{m}}^\ell$ of Proposition~\ref{prop:rho-bar}, we obtain that $\overline{\prol{m}}^\ell(\Delta)$ contains $\rP{m}(G)(\fm^\ell)/\rP{m}(G)(\fm^{\ell+1})$.

Since, by Lemma \ref{lem:connected-component}, $\rP{m}(G)$ is connected, Lemma \ref{lem:Zariski-dense} shows that $\prol{m}(\Delta)$ is Zariski-dense in $\rP{m}(G)$.

For the general case, let $G_0$ and $\rP{m}(G)_0$ denote the connected components of $1$ in $G$ and $\rP{m}(G)$ respectively. Let $g_1,\ldots, g_s\in \Delta$ be representatives for the residue classes of $G/G_0$. Then by Lemma \ref{lem:connected-component}, the elements $\prol{m}(g_1),\ldots, \prol{m}(g_s)\in \prol{m}(\Delta)$ are representatives for the residue classes of $\rP{m}(G)/\rP{m}(G)_0$.

By the proof for the connected case, $\prol{m}(\Delta)\cap \rP{m}(G)_0(\cK)$ is Zariski-dense in $\rP{m}(G)_0=\rP{m}(G_0)$, and since $\prol{m}(\Delta)$ contains elements in every connected component of $\rP{m}(G)$, we conclude that $\prol{m}(\Delta)$ is Zariski-dense in $\rP{m}(G)$.
\end{proof}

\section{Anderson \texorpdfstring{$t$}{t}-modules and motivic Galois groups}\label{S:t-modules-and-motivic-galois-groups}

 In this section, we will review $t$-modules, $t$-motives, and rigid analytic trivializations to the extent they are needed in this paper.
Afterwards in Section~\ref{S:Galois groups of prolongations}, we will apply our results from Sections \ref{S:prolongation-of-group-schemes} and \ref{S:Zdensity} to our objects of interest in this paper, namely, the motivic Galois groups associated to $t$-motives.

\subsection{Basic notation}

Let $\Fq$ be a finite field where $q$ is the positive power of a prime number $p$. Consider the polynomial ring $\Fq[\theta]$ in the variable $\theta$ and the corresponding fraction field $K=\Fq(\theta)$. Then,  $K_\infty=\Fq\ls{\frac{1}{\theta}}$ is the completion of $K$ with respect to the absolute value $\lvert \cdot \rvert_\infty$ determined by $\lvert \theta \rvert_\infty=q$, and we denote by $\CC_\infty$ the completion of a fixed  algebraic closure of $K_\infty$. Furthermore, let $\oK$ be the algebraic closure of $K$ inside $\CC_{\infty}$.

For a variable $t$ independent from $\theta$, set $\bA=\Fq[t]$, $\bK=\Fq(t)$. 
We let $\CC_\infty \langle t \rangle$ denote the Tate algebra comprising of power series converging on the closed unit disk  
\[
\CC_\infty \langle t \rangle = \bigg\{ \sum_{i \geq 0} e_it^i \in \CC_\infty [\![t]\!] \mid \lim_{i \rightarrow \infty} |e_i|_\infty =0\bigg\}.
\]
On the field of Laurent series $\CI\ls{t}$, we define for $n\in\ZZ$, the $n$-th Frobenius twist $\tau^n: \CI\ls{t} \rightarrow \CI\ls{t}$ by setting for $f = \sum_i a_it^i$,
\[
\tau^n ( f) := f^{(n)} = \sum_i a_i^{q^n}t^i.
\]
Note that the Frobenius twists restrict to endomorphisms on  the subrings $\oK$, $\CI$, and $\CI\cs{t}$.
For matrices $C=(c_{ij})$ with entries in $\CI\ls{t}$, we define the twisting by applying it entry-wise $C^{(n)} = (c_{ij}^{(n)})$.

For any intermediate field $K\subseteq L\subseteq \CI$, we define the twisted polynomial ring $L\{\tau\}$ 
subject to the condition $\tau c = c^q\tau$ for all $c \in L$. 
 Finally, for a matrix $B$ with entries in $L\{\tau\}$, let $\rd B$ denote the matrix  of constant terms of $B$, i.e., the matrix of coefficients of $\tau^0$.

\medskip

We mainly consider two types of hyperdifferential operators or hyperderivatives as defined in Example~\ref{Example:Hyperd}. One type is $(\hde{k}_\theta)_{k\geq 0}$ on $\cK=K$, $K_\infty$, and any separable algebraic extension thereof. The other is $(\hdte{k})_{k\geq 0}$ on various rings $\cR$ containing $\bA=\Fq[t]$, e.g.,~$\cR=\bA_\fp$ for some prime $\fp$ of $\bA$, $\CI\ls{t}$, $\CI[t]$, $\CI\cs{t}$, and $(\bA\otimes \CI)_\fp$ resp.~$(\bA\otimes L^\sep)_\fp$ (see Section \ref{Sub:RATsAndGaloisReps} for the latter two).
Using the hyperdifferential operators $(\hdte{k})_{k\geq 0}$
on such rings $\cR$, for $n \geq 0$, we will make use of the operator $\prol{n}:\Mat_{e}(\cR) \rightarrow \Mat_{e(n+1)}(\cR)$ defined by 
\[
\prol{n}(\Xi) = \begin{pmatrix}
			\Xi & \hdt{1}{\Xi} & \dots& \hdt{n}{\Xi} \\
			&    \ddots     & \ddots&\vdots\\
			&&\ddots&\hdt{1}{\Xi}\\
			&&&\Xi
		\end{pmatrix},
\]
where the hyperdifferential operator is meant to apply on $\Xi$ entry-wise.
We emphasize that the product rule for hyperdifferential operators implies that this operator is a ring homomorphism. 

It is no coincidence that we have the same symbol $\prol{n}$ here as in Formula \ref{eq:gammatohyper}, since it is the very same operator for invertible matrices.

\subsection{Abelian \texorpdfstring{$t$}{t}-modules and  \texorpdfstring{$t$}{t}-motives} \label{SubS:motive}

 Let $L$ be a field such that $K\subseteq L \subseteq \CC_\infty$. An \emph{Anderson $t$-module of dimension $d\geq 1$} defined over $L$ is a pair $(E, \phi)$ consisting of  an algebraic group $E$  over $L$ isomorphic to the $d$-dimensional additive algebraic group $\GG_{a/L}^d$ over $L$ and an $\Fq$-algebra homomorphism
\[
\phi: \bA=\Fq[t] \longrightarrow \End_{\FF_q}(E) \cong \Mat_{d\times d}(L\{\tau\}), 
\]
  such that 
  $(\rd\phi(t)-\theta\Id_d)^\ell \Lie(E)=\{0\}$ for some $\ell>0$. 
  This means that if we write $\phi(t)= \rd \phi(t)  + B_1 \tau + \dots + B_s\tau^s$, then $\rd \phi(t)  - \theta \Id_d$ is a nilpotent matrix. 
  A $t$-module of dimension $1$ such that $s\geq 1$ is called a \emph{Drinfeld module}.

The \emph{Anderson $t$-motive}  associated to a $t$-module $(E, \phi)$ of dimension $d$ defined over $L$ is defined as the left-$L[t]\{\tau\}$-module $\mot_E :=\Hom_{\Fq}(E, \GG_a) \cong \Mat_{1 \times d}(L\{\tau\})$, the module of $\Fq$-linear algebraic group morphisms, where for $m \in  \Mat_{1 \times d}(L\{\tau\})$, the $t$-action is given by
\[
t\cdot m = m \phi(t)\in  \Mat_{1 \times d}(L\{\tau\}). 
\]
Note that naturally $\mot_E$ is free and finitely generated as an $L\{\tau\}$-module. If, in addition, $\mot_E$ is finitely generated as an $L[t]$-module, then $\mot_E$ and $E$ are both said to be \emph{abelian} and the rank of $\mot_E$ as an $L[t]$-module\footnote{By \cite[Lemma 1.4.5]{ga:tm}, the $t$-motive $\mot_E$ is a free $L[t]$-module in this case.} is called the \emph{rank} of $\mot_E$ and $E$.

Associated to an Anderson $t$-module $(E, \phi)$ of dimension $d$ defined over $L$, there exists a unique power series of the form $\Exp_E= \Id_d \tau^0 + \sum_{i\geq 1} \alpha_i \tau^i \in \power{\Mat_d(L)}{\tau}$ satisfying
\[
\Exp_E \rd\phi(t) = \phi(t)\Exp_E.
\] 
Then, $\Exp_E: \Lie(E)(\CC_\infty) \rightarrow E(\CC_\infty)$ 
is an entire function, where 
\[
\Exp_E(\bsz) = \bsz + \sum_{i\geq 1} \alpha_i \bsz^{(i)},
\]
is called the \emph{exponential series of $E$}. If the function $\Exp_E$ is surjective, then $E$ is called a \emph{uniformizable} $t$-module. The kernel $\Lambda_E:= \ker(\Exp_E)$ of $\Exp_E$, called the \emph{period lattice of $E$}, is a free and finitely generated discrete $\bA$-submodule of $\Lie(E)(\CC_\infty)$. The elements of $\Lambda_E$ are called \emph{periods} of $E$. 
If, in addition, $E$ is abelian, then $\rk_{\bA}(\Lambda_E) = \rk_{L[t]}(\mot_E)$. 

For more information on these objects, we direct the reader to Anderson's original paper \cite{ga:tm} and expository sources \cite{db-mp:ridmtt}, \cite{dg:bsffa}, \cite{dt:gsff}.

\begin{Remark}
    In most parts of this paper, $L$ is allowed to be any intermediate field. However, we have to restrict to algebraic extensions of $K$ whenever Papanikolas' theorem is involved, and we have to restrict to separable extensions of $K$ whenever we are using hyperderivatives with respect to $\theta$. These restrictions will be included in the statements.
\end{Remark}

\begin{Example}\label{E:Drinfeld1CarlitzTensor1}
\begin{itemize} 
\item[(i)] \emph{Drinfeld modules.}  Let $(E_\varphi, \varphi)$ be a Drinfeld module of rank $r$ defined over $L$. 
Then, $\varphi(t)$ is of the form
\[
\varphi(t) = \theta + b_1\tau+\dots + b_r\tau^r \in L\{\tau\}, \quad b_r \neq 0.
\]
Note that $\{1, \tau, \dots, \tau^{r-1}\}$ forms an $L[t]$-basis of the $t$-motive $\mot_{E_\varphi}\cong L\{\tau\}$.

\item[(ii)] \emph{Tensor powers of the Carlitz module.} The Carlitz module $(E_C,C)$ is the Drinfeld module of rank $1$ defined over $K$ such that
\[
C_t = \theta +\tau. 
\]
Then, its associated abelian $A$-motive is $\mot_{E_C} = L\{\tau\}$ with $L[t]$-basis $\{1\}$. 
For $k \geq 2$, we consider the $k$-th tensor power $\mot_{E_C}^{\otimes k}$ over $K[t]$ of $\mot_{E_C}$. We can define a $t$-module $(G_k, C^{\otimes k})$ of dimension $k$ and rank $1$ called the $k$-th tensor power of the Carlitz module such that
\[
C^{\otimes k}(t)= \begin{pmatrix}
\theta & 1& 0 && 0\\
&\theta & 1&  \ddots&\vdots\\
&&\ddots&\ddots&0\\
&&&\ddots&1\\
&&&&\theta
\end{pmatrix} +\begin{pmatrix}
0&\dots &\dots &\dots& 0\\
\vdots& & &  &\vdots\\
\vdots&&&&\vdots\\
0&&&&\vdots\\
1&0&\dots&\dots&0
\end{pmatrix}\tau \in \Mat_k(L\{\tau\}).
\]
Its associated abelian $t$-motive $\mot_{G_k} \cong \Mat_{1\times k}(L\{\tau\})$ is isomorphic to $\mot_{E_C}^{\otimes k}$ and it is generated over $L[t]$ by $e_1=(1,0, \dots, 0)$ such that
\begin{equation}\label{E:Cntaut}
(t-\theta)^k e_1 = \tau e_1. 
\end{equation}
\end{itemize}
\end{Example}

\subsection{Rigid analytic trivializations and Galois groups}
\label{Sub:RATsAndGaloisGroups}

Let $(E, \phi)$ be an abelian $t$-module of rank $r$ defined over $L$. Suppose the entries of $\bsm = (m_1, \dots, m_r)^{\tr} \in \Mat_{r\times 1}(\mot_E)$ form an $L[t]$-basis of $\mot_E$ and let $\Theta_E  \in \Mat_r(L[t])$ denote the action of $\tau$ on $\bsm$
\[
\tau(\bsm) = \Theta_E \bsm.
\]
If there exists $\Upsilon \in \GL_r(\CI\cs{t})$ such that $\tau (\Upsilon \bsm)= \Upsilon \bsm$ i.e., 
\begin{equation}\label{eq:difference-equation-for-Upsilon}
\Upsilon = \Upsilon^{(1)}\Theta_E,
\end{equation}
then $\mot_E$ is said to be \emph{rigid analytically trivial} and $\Upsilon$ is called a \emph{rigid analytic trivialization of $\mot_E$} (see \cite{ga:tm}). Anderson  showed (see \cite[Theorem 4]{ga:tm}) that $\mot_E$ is rigid analytically trivial if and only if $E$ is uniformizable.

There are several Galois groups attached to a uniformizable abelian Anderson $t$-module $E$. Associated to a rigid analytic trivialization $\Upsilon$ of its $t$-motive $\mot_E$ as above (or more precisely to $\Upsilon^{-1}$), there is a \emph{difference Galois group} $\Gamma_\Upsilon$ which is a linear algebraic group over $\bK$. Namely, consider the ring $\cR:=L(t)[\Upsilon,\Upsilon^{-1}]\subseteq \operatorname{Frac}(\CI\cs{t})$ generated over $L(t)$ by the entries of $\Upsilon$ and of $\Upsilon^{-1}$, then for any $\bK$-algebra $\sfA$, the $\sfA$-points of $\Gamma_\Upsilon$ are given by
\[  \Gamma_\Upsilon(\sfA) := \left\{ \alpha\in \Aut_{L(t)\otimes_{\bK}\sfA}(\cR\otimes_{\bK} \sfA) \,\middle|\, (\tau\otimes \id)\circ \alpha = \alpha\circ (\tau\otimes \id) \right\},\]
the group of $(L(t)\otimes_{\bK}\sfA)$-algebra automorphisms of $\cR\otimes_{\bK} \sfA$ that commute with the endomorphism $\tau\otimes \id$. One obtains an explicit embedding of linear algebraic groups $\Gamma_\Upsilon$ into $\GL_{r,\bK}$ via
\begin{equation}\label{eq:embedding-of-Gamma_Upsilon}
      C_\Upsilon: \Gamma_\Upsilon(\sfA)\to \GL_r(\sfA), \alpha \mapsto C_\Upsilon(\alpha):=\Upsilon\cdot \alpha(\Upsilon)^{-1}      
\end{equation} 
(see e.g.~\cite[Appendix A]{qg-am:ctmcrzvp}).

\begin{Remark}\label{rem:rats}
    Note that a rigid analytic trivialization $\Upsilon$ is not unique. All rigid analytic trivializations (with respect to the same basis $\bsm$) are obtained from $\Upsilon$ as $D\cdot \Upsilon$ with $D\in \GL_r(\bK)$. On the other hand, a different choice of basis $\bsm$ would change $\Upsilon$ to $\Upsilon\cdot  B$ for some matrix $B\in \GL_r(L[t])$.
 This implies that the group $\Gamma_\Upsilon$ does not depend on the specific chosen $\Upsilon$, but different choices lead to embeddings $C_\Upsilon$ that are conjugate by matrices in $\GL_r(\bK)$ (a possible base change matrix $B\in \GL_r(L[t])$ cancels out in the formula).
\end{Remark}

A second Galois group is obtained from the $t$-motive more abstractly. The $t$-motive $\mot_E$ can be seen as an object in a Tannakian category over $\bK$, and hence there is the Tannakian Galois group $\Gamma_{\mot_E}$ associated to $\mot_E$, called the \emph{motivic Galois group} of $\mot_E$ (see \cite[Definition 2.3.28]{uh-akj:pthshcff}).

Furthermore, attached to $E$, there is also a \emph{$t$-comotive} $\dmot_E$ (formerly called dual $t$-motive) which is \emph{coabelian/$t$-finite} by \cite[Theorem A]{am:aefam}, and a rigid analytic trivialization $\Psi\in \GL_r(\CI\cs{t})$ for $\dmot_E$.
Similar to the $t$-motive case, there is a motivic Galois group $\Gamma_{\dmot_E}$ for $\dmot_E$ and a difference Galois group $\Gamma_{\Psi}$ for $\Psi$.

Due to the (anti-)equivalence of categories between these $t$-motives and $t$-comotives (see \cite[Prop.~2.4.17]{uh-akj:pthshcff}), the Galois group $\Gamma_{\dmot_E}$ is isomorphic to 
$\Gamma_{\mot_E}$. Furthermore, a rigid analytic trivialization $\Psi$ of $\dmot_E$ can be obtained from the rigid analytic trivialization $\Upsilon$ of $\mot_E$ via
\begin{equation}\label{E:ratcomotive}
\Psi=(\Upsilon^{\tr})^{(1)} 
\end{equation}
when choosing an appropriate basis of $\dmot_E$.
This shows that $L(t)[\Psi,\Psi^{-1}]=L(t)[\Upsilon,\Upsilon^{-1}]$ and that the corresponding difference Galois group $\Gamma_{\Psi}$ equals $\Gamma_{\Upsilon}$. Finally by \cite[Theorem 4.5.10]{mp:tdadmaicl}, also the Galois groups $\Gamma_{\dmot_E}$ and $\Gamma_{\Psi}$ are isomorphic over $\bK$, and hence all four Galois groups $\Gamma_{\Upsilon}$, $\Gamma_{\mot_E}$ $\Gamma_{\Psi}$ and $\Gamma_{\dmot_E}$ are isomorphic.

\medskip

In what follows, we will not use the $t$-comotive at all. Instead, we will mainly talk about $\Gamma_{\mot_E}$, and tacitly use the isomorphism to $\Gamma_{\Upsilon}$ and its explicit description as an algebraic subgroup of $\GL_{r,\bK}$. The latter  will also be used in the description of the Galois representations. The above connections allow us to restate Papanikolas' Theorem (which is given for $\Psi$ and $\Gamma_{\dmot_E}$) in terms of $\Upsilon$ and $\Gamma_{\mot_E}$.

\begin{Theorem}[{Papanikolas' Theorem; see~\cite[Thm.~1.1.7]{mp:tdadmaicl}}] \label{T:Pap}
Let $\mot_E$ be a rigid analytically trivial abelian $t$-motive defined over an algebraic extension $L$ of $K$, and let $\Gamma_{\mot_E}$ denote its motivic Galois group. Suppose 
$\Upsilon \in \GL_r(\CC_\infty \langle t \rangle)$ is a rigid analytic trivialization of $\mot_E$. 
Then,
\[
\trdeg_{L} L(\Upsilon^{(1)}|_{t=\theta}) = \trdeg_{L(t)} L(t)(\Upsilon^{(1)}) = \trdeg_{L(t)} L(t)(\Upsilon) = \dim \Gamma_{\mot_E}. 
\]
\end{Theorem}

Note that by \cite[Prop.~3.1.3]{ga-wb-mp:darasgvpc}, the entries of  $\Psi=(\Upsilon^{\tr})^{(1)}$ converge for all elements in $\CI$, and hence are regular at $t=\theta$, so that the specialization of the entries of $\Psi^{\tr}=\Upsilon^{(1)}$ at $t=\theta$, denoted by $\Upsilon^{(1)}|_{t=\theta}$, is meaningful.

\subsection{Prolongation of \texorpdfstring{$t$}{t}-modules and \texorpdfstring{$t$}{t}-motives}\label{Sub:Promodulesmotives}

To introduce prolongations of $t$-modules and $t$-motives, we use the hyperdifferential operator $(\hde{k}_t)_{k\geq 0}$ on $L[t]$.

As defined in \cite{am:ptmaip}, for $n \geq 0$, the \emph{$n$-th prolongation $\rho_n\mot_E$} of the $t$-motive $\mot_E$ defined over $L$ is the $L[t]\{\tau\}$-module generated by symbols $D_im$ for $m \in \mot_E$ and $0\leq i \leq n$ subject to the relations
\begin{itemize}
\item[(a)] $D_i(m_1+m_2) = D_i(m_1)+D_i(m_2)$,
\item[(b)] $D_i(a\cdot m) = \sum_{i=i_1+i_2} \hde{i_1}_{t}(a)\cdot D_{i_2}m$,
\item[(c)] $\tau(a\cdot D_i m) = a^{(1)} \cdot D_i(\tau m)$,
\end{itemize}
where $m, m_1, m_2 \in \mot_E$ and $a \in L[t]$.

Suppose that the $t$-module $(E, \phi)$ with $E \cong \GG_{a/L}^d$ where $d\geq 1$ is given by the $\Fq$-linear algebra homomorphism $\phi: \bA \rightarrow \Mat_d(L\{\tau\})$. Then, by \cite[Theorem~5.2]{am:ptmaip}, the \emph{$n$-th prolongation of $E$} is the $t$-module $(\rho_nE, \rho_n\phi)$ of dimension $(n+1)d$ with $\rho_n E \cong \GG_{a/L}^{(n+1)d}$ and is given by $\rho_n\phi: \Fq[t] \rightarrow \Mat_{(n+1)d}(L\{\tau\})$ where
\[
\rho_n\phi(t)= \begin{pmatrix}
\phi(t) & 0 & \dots & 0 \\
-\Id_d &\ddots &\ddots & \vdots\\
\vdots&\ddots &\ddots &  0\\
0 & \dots & -\Id_d & \phi(t)
\end{pmatrix}.
\]
One can show that the $t$-motive $\mot_{\rho_n E}$ associated to $(\rho_nE, \rho_n\phi)$ is isomorphic to $\rho_n \mot_E$ (see  \cite[Theorem~5.2]{am:ptmaip} or \cite[Proposition~5.2.11]{cn-mp:hpqpam}). 
Note that for $n \geq 0$, 
$\mot_{\rho_n E}$ is abelian if $\mot_E$ is abelian. 

For uniformizability, we have the following theorem.

\begin{Theorem}[{\cite[Proof of Theorem~3.6(a)]{am:ptmaip}}]\label{T:RATpro}
 Let $E$ be an abelian $t$-module, $\mot=\mot_E$ be its corresponding $t$-motive, and for $n \geq 0$, $\rho_n \mot$ be its $n$-th prolongation. If $\Upsilon \in \GL_{r}(\CI\cs{t})$ is a rigid analytic trivialization of $\mot$, then $\prol{n}(\Upsilon) \in \GL_{r(n+1)}(\CI\cs{t})$ is a rigid analytic trivialization of $\rho_n \mot$. 
\end{Theorem}

\subsection{Rigid analytic trivializations, periods and quasi-periods}\label{SS:RatPeriodsTractable}

The theory of biderivations and quasi-periodic functions for Drinfeld modules  was developed by Anderson, Deligne, Gekeler, and Yu (see \cite{MR1018059}, \cite{jy:opaqpodm})) and was extended to the setting of Anderson $t$-modules by Brownawell and Papanikolas (see \cite{db-mp:ligvpc}).
Let $\{\lambda_1, \dots, \lambda_r\}$ be an $\bA$-basis of the period lattice $\Lambda_{E}$ of an abelian $t$-module $(E, \phi)$ of dimension $d$ and rank $r$. Then, for $1\leq j \leq r$, there exist quasi-periods $F_{j}(\lambda_i)$ for each $1\leq i \leq r$. The $L$-span of these quasi-periods is independent of the choices of the $\lambda_i$'s and the $F_{j}(\lambda_i)$'s.

 If $E$ is a Drinfeld module, then one can pick $F_r$ such that $F_r(\lambda_i)=\lambda_i$ and show that $\Span_{L}(\Lambda_E) \subseteq \Span_L(\Upsilon_E^{(1)}|_{t=\theta})$ (see \cite[\S4.2]{fp:aiacnn} and \cite[Prop.~3.4.7(c)]{cc-mp:aipldm}). This is not the case for higher dimensional $t$-modules as the situation depends on the nilpotent matrix $\rd\phi_t-\theta\Id_d$, which provides only the tractable coordinates of the periods of $E$.  If $\rd\phi_t$ is in Jordan canonical form, then for each $\lambda\in \Lambda_E$, the \emph{tractable} coordinates of $\lambda$ are those coordinates that lie at the bottom of a Jordan block of $\rd\phi_t$. 
Since the $L$-span of the quasi-periods of $E$ is unique up to isomorphisms of $E$, we may assume that $\rd\phi_t$ is in Jordan canonical form.

 In \cite{am:ptmsv}, the first author showed that the entries of a $\prol{n}(\Upsilon_E^{(1)})$ can be related to all coordinates of the periods of the $d$-dimensional $t$-module $E$ in a nice way. Then,  in \cite{cn-mp:hpqpam}, Papanikolas and the second author showed that the entries of $\prol{n}(\Upsilon_E^{(1)})$ for sufficiently large $n \geq 0$ can be related to the hyperderivatives with respect to $\theta$ of periods and quasi-periods also. Since $(\hde{k}_\theta)_{k\geq 0}$ on $K$ and $K_{\infty}$ extend to separable extensions, but not to inseparable extensions, we let $L$ be a separable extension of $K$.

\begin{Theorem}[see~{\cite[Main Theorem]{am:ptmsv}, \cite[Theorem F]{cn-mp:hpqpam}}]\label{T:hyperRATnPer}
 Let $(E,\phi)$ be a uniformizable abelian  $t$-module of rank $r$ and dimension $d$ defined over a finite separable extension $L$ of $K$.  
 Let $j \geq 0$ and choose $n\geq 0$ so that $(d\phi(t)-\theta \Id_d)^{n-j}=0$.
 If $\lambda_1, \dots, \lambda_r$ is an $\bA$-basis of $\Lambda_E$, then
 \[
 \Span_{L}\left(\hde{j}_\theta(\lambda_i)\right) \subseteq  \Span_{L}\left(\prol{n}(\Upsilon_E^{(1)}) \mid_{t=\theta} \right). 
 \]
 Moreover, the $L$-linear combinations of entries of $\prol{n}(\Upsilon_E^{(1)}) \mid_{t=\theta}$ generate all quasi-periods and tractable coordinates of periods of $E$, and their hyperderivatives upto $n$. 
\end{Theorem}

\subsection{Rigid analytic trivializations and Galois representations}\label{Sub:RATsAndGaloisReps}

For $0\neq \mathfrak{a} \in \bA$ and an abelian $t$-module $(E, \phi)$, we consider the finite kernel 
$\phi[\mathfrak{a}]:= \ker(\phi(\mathfrak{a})) = \{ \boldsymbol{x} \in \CC_\infty^d \ \mid \ \phi(\mathfrak{a})\boldsymbol{x}=0\}$ which is isomorphic to $(\bA/(\mathfrak{a}))^{\oplus r}$ (see \cite[Thm.~7.2.1]{dt:ffa}). We define the \emph{$\fp$-adic Tate module} $T_\fp(E)$ for a non-zero prime $\fp \in \bA$ to be
\[
T_{\fp}(E) := \lim_{\substack{\longleftarrow \\ n}} \phi[\fp^{n+1}] \cong \bA_\fp^{\oplus r}, 
\]
where $\bA_\fp$ denotes the completion of $\bA$ at $\fp$. 
The elements of $\phi[\mathfrak{a}]$ are separable over $L$ i.e., $\phi[\mathfrak{a}] \subseteq E(L^{\sep})$, where $L^{\sep}$ is the separable closure of $L$ inside $\CC_{\infty}$. The elements of the Galois group $\Gal(L^\sep/L)$ act on $T_{\fp}(E)$ inducing automorphisms on it and this defines the representation
\[
\varrho_{E,\fp}:\Gal(L^\sep/L)\to\GL_r(\bA_\fp),\]
called the $\fp$-adic Galois representation associated to $E$. 

\begin{Theorem}[{\cite[Theorem 9.10]{am:naamcigp}}] \label{thm:Galois-representation}
Let $(E, \phi)$ be a uniformizable abelian $t$-module of rank $r$ defined over $L$.
Let $\Upsilon\in \GL_{r}(\CI\cs{t})$ be a rigid analytic trivialization of the $t$-motive $\mot_E$, and $\Gamma_\Upsilon\cong \Gamma_{\mot_E}$ be the corresponding Galois groups. Consider $\Gamma_\Upsilon\subseteq \GL_{r,\bK}$ via the embedding $C_\Upsilon$ of \eqref{eq:embedding-of-Gamma_Upsilon}. For a non-zero prime $\fp$ of $\bA$, let $\bK_{\fp}$ denote the field of fractions of $\bA_{\fp}$. Then the $\fp$-adic Galois representation $\varrho_{E,\fp}:\Gal(L^\mathrm{sep}/L)\to \GL_r(\bA_\fp)$ attached to $E$ is such that
\[ 
\varrho_{E,\fp}(\Gal(L^\mathrm{sep}/L)) \subseteq \Gamma_{\Upsilon} (\bA_{\fp}) := \GL_r(\bA_\fp) \cap \Gamma_{\Upsilon}(\bK_\fp).
\]
\end{Theorem}

Theorem \ref{thm:Galois-representation} is made more explicit in \cite[Rem.~2.18]{am:naamcigp}: For a non-zero prime ideal $\fp$ of $\bA$, let $\up$ be the monic generator of $\fp$, and let $(\bA\otimes\CI)_\fp$ be the $\fp$-adic completion $\varprojlim\limits_n(\bA/\fp^{n+1}\otimes \CC_\infty)$. 
Let $\D_\fp:\CI\cs{t}\to (\bA\otimes\CI)_\fp$ be the Taylor series expansion at $\fp$, i.e.,~
\[  \D_{\fp}(h):= \sum_{n=0}^\infty \up^n \left( \partial_{\up}^{(n)}(h)\textrm{ mod }\fp\right). \]
Note, by \cite[Theorem 9.6]{am:naamcigp}, the coefficients of $\up^n$ in $\D_\fp(\Upsilon)$ are indeed in $\Fpp\otimes L^{\sep}$, and for $\gamma\in \Gal(L^\sep/L)$, we have
\begin{equation}\label{eq:explicit-Galois-represenation}
     \varrho_{E,\fp}(\gamma) = \D_\fp(\Upsilon)\cdot \gamma(\D_\fp(\Upsilon))^{-1}. 
\end{equation}

\section{Galois groups of prolongations and transcendence}\label{S:Galois groups of prolongations}

In this section, we determine Galois groups of prolongations of uniformizable abelian $t$-modules, and specifically determine their dimensions. We further apply this to obtain transcendence results for quasi-periods and coordinates of periods.

\subsection{Galois group of prolongations}

For a uniformizable abelian $t$-module $(E, \phi)$ defined over $L$, we use the notations $\mot_E$, $\Gamma_{\mot_E}$, $\Upsilon_E$, to represent the $t$-motive of $E$, its motivic Galois group, and a rigid analytic trivialization of $\mot_E$, respectively. Furthermore, $\varrho_{E,\fp}:\Gal(L^\sep/L)\to \GL_r(\bA_\fp)$ represents the $\fp$-adic Galois representation of $E$ with respect to $\Upsilon_E$ as described in the lines after Theorem \ref{thm:Galois-representation}. Accordingly, for its prolongations $\rho_n E$ ($n\geq 0$), we have the corresponding objects 
$\mot_{\rho_n E}=\rho_n\mot_E$, $\Gamma_{\rho_n\mot_E}$, $\prol{n}(\Upsilon)$, and $\varrho_{\rho_nE,\fp}:\Gal(L^\sep/L)\to \GL_{r(n+1)}(\bA_\fp)$. Further, recall the notion of the $n$-th prolongation $\rP{n}(\Gamma_{\mot_E})$ of 
$\Gamma_{\mot_E}$ from Definition~\ref{D:proAGS}.

\medskip

The first goal is to give an upper bound on the Galois group $\Gamma_{\rho_n\mot_E}$.

Let $\fracTate=\operatorname{Frac}(\CI\cs{t})$ be the field of fractions of the Tate algebra. For two matrices $C,D\in \GL_s(\fracTate)$, we denote by $C\boxtimes D$ the matrix in $\GL_s(\fracTate\otimes_{L(t)}\fracTate)$ with $(i,j)$-entry $\sum_{l=1}^s C_{il}\otimes D_{lj}$.

\begin{Lemma}\label{lem:description-of-Gamma-Upsilon}
    Let $(E,\phi)$ be a uniformizable abelian $t$-module of rank $r$ defined over $L$, and $\Upsilon$ a rigid analytic trivialization of the $t$-motive $\mot_E$.
    The difference Galois group $\Gamma_\Upsilon$ is the smallest closed subgroup $G$ of $\GL_{r,\bK}$ defined over $\bK$ such that $\Upsilon\boxtimes \Upsilon^{-1}\in G(\fracTate\otimes_{L(t)} \fracTate)$.
\end{Lemma}

\begin{proof}
In \cite[\S 4.2.1]{mp:tdadmaicl}, Papanikolas showed that the difference Galois group $\Gamma_\Psi$ is the smallest closed subgroup of $\GL_{r,\bK}$ containing  $\Psi^{-1}\boxtimes \Psi$, where $\Psi$ is a rigid analytic trivialization of the $t$-comotive.
Our embedding of $\Gamma_\Upsilon$ into $\GL_{r,\bK}$, however corresponds to the matrix 
$\Upsilon^{-1}=((\Psi^{-1})^{\tr})^{(-1)}$ which is obtained from the embedding used for $\Gamma_\Psi$ by the homomorphism $X\mapsto (X^{-1})^{\tr}$.
So we get from Papanikolas' result that $\Gamma_\Upsilon$ is the smallest closed subgroup containing $\Upsilon^{(1)}\boxtimes (\Upsilon^{-1})^{(1)}=\left(\Upsilon\boxtimes \Upsilon^{-1}\right)^{(1)}$.
From the difference equation \eqref{eq:difference-equation-for-Upsilon} for $\Upsilon$, we see that $\left(\Upsilon\boxtimes \Upsilon^{-1}\right)^{(1)}=\Upsilon\boxtimes \Upsilon^{-1}$, since the tensor product is over $L(t)$.
This concludes the proof.
\end{proof}

\begin{Theorem}\label{thm:upper-bound-on-group-of-prolongation}
    Let $(E,\phi)$ be a uniformizable abelian $t$-module and $\mot=\mot_E$ be its corresponding $t$-motive and for $n \geq 0$, $\rho_n \mot$ be its $n$-th prolongation.
    Then the Galois group of the prolongation, $\Gamma_{\rho_n \mot}$, is contained in $\rP{n}(\Gamma_\mot)$, the prolongation of the Galois group of $\mot$,  as a subgroup.
\end{Theorem}

 \begin{proof}
    Let $\Upsilon:=\Upsilon_E$ be a rigid analytic trivialization of $\mot=\mot_E$.
    By Theorem \ref{T:RATpro}, the matrix $\prol{n}(\Upsilon)$ is a rigid analytic trivialization for $\rho_n \mot$, and by Lemma \ref{lem:description-of-Gamma-Upsilon} applied to $\rho_n E$, the Galois group $\Gamma_{\rho_n \mot}=\Gamma_{\prol{n}(\Upsilon)}$ is the smallest closed $\bK$-subgroup $G$ of $\GL_{r(n+1),\bK}$ such that $\prol{n}(\Upsilon)\boxtimes \prol{n}(\Upsilon)^{-1}\in G(\fracTate\otimes_{L(t)}\fracTate)$.
  
    Since $\Upsilon\boxtimes\Upsilon^{-1}\in \Gamma_\Upsilon(\fracTate\otimes_{L(t)}\fracTate)$ (again by Lemma \ref{lem:description-of-Gamma-Upsilon}), the map $\prol{n}$ provides that
    \[ \prol{n}(\Upsilon\boxtimes\Upsilon^{-1})\in \rP{n}(\Gamma_\Upsilon)(\fracTate\otimes_{L(t)}\fracTate).\]
    However, since $\prol{n}$ is a homomorphism of algebras, we have
    \[ \prol{n}(\Upsilon\boxtimes\Upsilon^{-1})=\prol{n}(\Upsilon)\boxtimes \prol{n}(\Upsilon)^{-1}.\]
    So, $\prol{n}(\Upsilon)\boxtimes \prol{n}(\Upsilon)^{-1}\in \rP{n}(\Gamma_\Upsilon)(\fracTate\otimes_{L(t)}\fracTate)$, and hence $\Gamma_{\rho_n\mot}\subseteq \rP{n}(\Gamma_\Upsilon)=\rP{n}(\Gamma_\mot)$.
 \end{proof}

For the lower bound, we need some lemmas on the $\fp$-adic Galois representation of a prolongation.

\begin{Lemma}\label{lem:repr-of-prol}
    Let $(E,\phi)$ be a uniformizable abelian $t$-module over $L$, let $n\geq 0$, and let $\fp$ be a non-zero prime of $\bA$. The $\fp$-adic Galois representation 
    \[ \varrho_{\rho_nE,\fp}:\Gal(L^\sep/L)\to \GL_{r(n+1)}(\bA_\fp) \] is given by
    \[ \varrho_{\rho_nE,\fp} = \prol{n}\circ \varrho_{E,\fp}. \]
\end{Lemma}

\begin{proof}
    Using Theorem \ref{T:RATpro} and Equation \eqref{eq:explicit-Galois-represenation}, our task is to show that for all $\gamma\in \Gal(L^\sep/L)$,
    \begin{equation}\label{eq:identity-for-gal-rep}
        \prol{n}\left( \D_\fp(\Upsilon)\cdot \gamma(\D_\fp(\Upsilon))^{-1} \right) = \D_\fp\bigl(\prol{n}(\Upsilon)\bigr)\cdot \gamma\left(\D_\fp\bigl(\prol{n}(\Upsilon)\bigr)\right)^{-1}.
    \end{equation}
    Since $\prol{n}$ is a ring homomorphism, this boils down to showing that $\prol{n}$ commutes with $\D_\fp$ and with $\gamma$, i.e., that the following two diagrams commute:
    
\[    \xymatrix{ \Mat_r(\CI\cs{t}) \ar[r]^(.47){\D_\fp} \ar[d]^{\prol{n}} &  \Mat_r((\bA\otimes \CI)_\fp) \ar[d]^{\prol{n}} \\
    \Mat_{r(n+1)}(\CI\cs{t}) \ar[r]^(.45){\D_\fp} & \Mat_{r(n+1)}((\bA\otimes \CI)_\fp)        
    }\]
    
    and
    
    \[ \xymatrix{ \Mat_r((\bA\otimes L^\sep)_\fp) \ar[r]^{\gamma} \ar[d]^{\prol{n}} &  \Mat_r((\bA\otimes L^\sep)_\fp) \ar[d]^{\prol{n}} \\
    \Mat_{r(n+1)}((\bA\otimes L^\sep)_\fp) \ar[r]^{\gamma} & \Mat_{r(n+1)}((\bA\otimes L^\sep)_\fp)        
    }\]
    
For the first diagram, we recognize that by \cite[Remark 2.18 \& Prop.~2.3]{am:naamcigp}, the map $\D_{\fp}$ is the continuous extension (with respect to the Gauss norm) of the inclusion $\CI[t]=\bA\otimes \CI \hookrightarrow (\bA\otimes \CI)_\fp$ which is compatible with $\prol{n}$ by construction. As $\CI[t]$ is dense in $\CI\cs{t}$, the continuous extension is unique, and hence the whole diagram commutes.

The second diagram clearly commutes when we restrict to matrices with coefficients in $\bA\otimes L^\sep$, since $\prol{n}$ comprises of hyperderivatives $\hdte{k}$ which are $L^\sep$-linear extensions of the hyperderivatives on $\bA$, whereas the action of $\gamma$ is the $\bA$-linear extension of the natural action on $L^\sep$. 
As $\bA\otimes L^\sep$ is dense in the $\fp$-adic completion $(\bA\otimes L^\sep)_\fp$, and both $\gamma$ and $\prol{n}$ are extended continuously, also the given diagram commutes.
\end{proof}

\begin{Lemma}\label{lem:dense-prolongation-of-representation}
	Let $(E,\phi)$ be a uniformizable abelian $t$-module over $L$, and let $n\geq 0$. Further, let $\fp$ be a non-zero prime of $\bA$, and $G\subseteq \Gamma_{\mot_E}$ be a closed subgroup defined over $\bK$.
    If $\image(\varrho_{E,\fp})\cap G(\bK_\fp)$ is $\fp$-adically open in $G(\bK_\fp)$, then $\image(\varrho_{\rho_nE,\fp})\cap \rP{n}(G)(\bK_\fp)$
    is Zariski-dense in $\rP{n}(G)$.
\end{Lemma}

\begin{proof}
Let $\Delta:=\varrho_{E,\fp}(\Gal(L^\sep/L))\cap G(\bK_\fp)$.
By Lemma \ref{lem:repr-of-prol}, $\varrho_{\rho_nE,\fp}=\prol{n}\circ \varrho_{E,\fp}$, and hence, 
\begin{equation*}
    \image(\varrho_{\rho_nE,\fp})\cap \rP{n}(G)(\bK_\fp) =
    \prol{n}\left( \image(\varrho_{E,\fp})\right)\cap \rP{n}(G)(\bK_\fp)
    = \prol{n}\left( \image(\varrho_{E,\fp})\cap G(\bK_\fp)\right)=\prol{n}(\Delta).
\end{equation*} 
The claim now follows from Theorem \ref{thm:Galois-representation}.
\end{proof}

\begin{Theorem}\label{T:galois-group-of-prolongation}
	Let $(E,\phi)$ be a uniformizable abelian  $t$-module over $L$, let $n\geq 0$, and denote by $\Gamma_{\mot_E, 0}$ the connected component of $1$ in $\Gamma_{\mot_E}$.
    Assume that there exists a non-zero prime $\fp\subset \bA$ such that $\image(\varrho_{E,\fp})\cap \Gamma_{\mot_E, 0}(\bK_\fp)$ is $\fp$-adically open in $\Gamma_{\mot_E, 0}(\bK_\fp)$. Then
    \[
    \Gamma_{\rho_{n}\mot_E} = \rP{n}(\Gamma_{\mot_E}),
    \]
	where $\Gamma_{\rho_{n}\mot_E}$ is the Galois group of the $n$-th prolongation $\rho_{n}\mot_E$ of $\mot_E$ and $\rP{n}(\Gamma_{\mot_E})$  is the $n$-th prolongation of the Galois group of $\mot_E$.
	In particular, $\dim(\Gamma_{\rho_{n}\mot_E})=(n+1)\cdot \dim(\Gamma_{\mot_E})$.
\end{Theorem}

\begin{proof}
By Theorem~\ref{thm:upper-bound-on-group-of-prolongation}, we already have
$\Gamma_{\rho_{n}\mot_E}\subseteq \rP{n}(\Gamma_{\mot_E})$, and by  Lemma \ref{lem:dense-prolongation-of-representation} applied to $G=\Gamma_{\mot_E,0}$, we have that 
$\image(\varrho_{\rho_nE,\fp})\cap \rP{n}(\Gamma_{\mot_E,0})(\bK_\fp)$ is Zariski-dense in $\rP{n}(\Gamma_{\mot_E,0})=\rP{n}(\Gamma_{\mot})_0$.
Since by Theorem~\ref{thm:Galois-representation}, this image is also contained in $\Gamma_{\rho_{n}\mot_E}$, we conclude $\Gamma_{\rho_{n}\mot_E}\supseteq \rP{n}(\Gamma_{\mot})_0$.

However, the projection $\rho_{n}\mot_E\to \mot_E$ induces a surjection $\Gamma_{\rho_{n}\mot_E}\to \Gamma_{\mot_E}$ which implies that the number of connected components of $\Gamma_{\rho_{n}\mot_E}$ has to be at least the number of connected components of $\Gamma_{\mot_E}$. Since by Lemma \ref{lem:connected-component}, $\rP{n}(\Gamma_{\mot_E})$ has the same number of connected components as $\Gamma_{\mot_E}$, we obtain
$\Gamma_{\rho_{n}\mot_E}=\rP{n}(\Gamma_{\mot_E})$.

The formula for the dimension is now simply a consequence of Proposition~\ref{prop:ses-and-dim-of-prolongations}.
\end{proof}

From what we proved in this section, we also get more evidence for the Mumford-Tate conjecture (see Subsection~\ref{Sub:BMGG}).

\begin{Theorem}\label{thm:evidence-Mumford-Tate-conjecture}
    Let $(E, \phi)$ be a uniformizable abelian $t$-module over a finite extension $L$ of $K$ and let $n \geq 0$. Assume that for every non-zero prime $\fp\subset \bA$, the image of the $\fp$-adic Galois representation $\varrho_{E,\fp}$ is $\fp$-adically open in $\Gamma_{\mot_E}(\bK_\fp)$.
    Then, the Mumford-Tate conjecture holds for the $n$-th prolongation $\rho_nE$ of $E$, i.e., for every non-zero prime $\fp\subset \bA$, the image of $\varrho_{\rho_nE,\fp}$ is Zariski-dense in $\Gamma_{\rho_n\mot_E}$.
\end{Theorem}

\begin{proof}
    If $\image(\varrho_{E,\fp})$ is $\fp$-adically open in $\Gamma_{\mot_E}(\bK_\fp)$, then $\image(\varrho_{E,\fp})\cap \Gamma_{\mot_E,0}(\bK_\fp)$ is $\fp$-adically open in $\Gamma_{\mot_E,0}(\bK_\fp)$. So we conclude from Theorem \ref{T:galois-group-of-prolongation}, that $\Gamma_{\rho_n\mot_E}=\rP{n}(\Gamma_{\mot_E})$. Lemma \ref{lem:dense-prolongation-of-representation} applied to $G=\Gamma_{\mot_E}$ gives the desired conclusion.
\end{proof}

\subsection{Transcendence results}

As mentioned earlier, results on the dimension of Galois groups give us transcendence results on the quasi-periods  and the tractable coordinates of periods, and their hyperderivatives. Bear in mind how these elements are obtained from specializing rigid analytic trivializations (see Theorem \ref{T:Pap} and Theorem \ref{T:hyperRATnPer}). As previously mentioned, since the hyperdifferential operators $(\hde{k}_\theta)_{k\geq 0}$ on $K$ and $K_{\infty}$ extend to separable extensions, but not to inseparable extensions, we let $L$ be a separable extension of $K$ whenever we consider hyperderivatives with respect to $\theta$ of periods and quasi-periods.

\begin{Theorem}\label{Thm:transcendence-degrees}
Let $(E,\phi)$ be a uniformizable abelian $t$-module over a finite extension $L$ of $K$, and let $n \geq 0$.   
   Further, let 
	$\fp$ be a non-zero prime of $\bA$ and $\varrho_{E,\fp}:\Gal(L^{\mathrm{sep}}/L)\to \GL_r(\bA_\fp)$ the $\fp$-adic Galois representation associated to $E$.
	If the image of $\varrho_{E,\fp}$ is $\fp$-adically open in $\Gamma_{\mot_E}(\bK_\fp)$, then 
	\begin{equation}\label{E:tr1}
	\trdeg_{L}L(\prol{n}(\Upsilon_E^{(1)}) \mid_{t=\theta}) = (n+1)\cdot \trdeg_{L}L(\Upsilon_E ^{(1)}|_{t=\theta}). 
	\end{equation}
	In particular, let $L$ be a finite separable extension of $K$ and let $\mathcal{G}_n$ denote the field generated over $L$ by all quasi-periods  and tractable coordinates of the periods of $E$, and their hyperderivatives upto $n$. Then,
    \begin{equation}\label{E:tr1Q}
        \trdeg(\mathcal{G}_n:L) = (n+1)\cdot \trdeg(\mathcal{G}_0:L). 
          \end{equation} 
\end{Theorem}

\begin{proof}
    Equation \eqref{E:tr1} follows from Theorems~\ref{T:Pap} and \ref{T:galois-group-of-prolongation}. Then, \eqref{E:tr1Q} follows from Theorem~\ref{T:hyperRATnPer}. 
\end{proof}

Using a significant property of rigid analytic trivializations of uniformizable abelian $t$-modules, we have the following result. 
\begin{Theorem}\label{Thm:periodentries}
Let $(E,\phi)$ be a uniformizable abelian $t$-module of dimension $d$ and rank $r$ defined over a finite extension $L$ of $K$.
   Further, let 
	$\fp$ be a non-zero prime of $\bA$ and $\varrho_{E,\fp}:\Gal(L^{\mathrm{sep}}/L)\to \GL_r(\bA_\fp)$ the $\fp$-adic Galois representation associated to $E$. Suppose that the image of $\varrho_{E,\fp}$ is $\fp$-adically open in $\Gamma_{\mot_E}(\bK_\fp)$. If the tractable coordinates of periods of $E$ are algebraically independent over $L$, then all coordinates of these periods are algebraically independent over $L$.
\end{Theorem}

\begin{proof}
    After choosing a basis of $\dmot_E$ such that the conditions of \cite[Prop.~3.28]{cn-mp:hpqpam} are satisfied, let $\Psi$ be the corresponding rigid analytic trivialization of $\dmot_E$. By Remark~\ref{rem:rats} and \eqref{E:ratcomotive}, $\Psi = \mathfrak{C}\cdot(\Upsilon^{\tr})^{(1)}$ for some $\mathfrak{C}\in \GL_r(L[t])$.  
    Let $\lambda$ be a non-zero period of $E$ with $m \leq d$ of the coordinates of $\lambda$ tractable. Note that we have $m \leq r$ (see \cite[Rem.~3.30]{cn-mp:hpqpam}).  Then, \cite[Thm.~2.5.32]{uh-akj:pthshcff} (see also \cite[Prop.~3.18(b)]{cn-mp:hpqpam}) imply that by \cite[Prop.~3.28]{cn-mp:hpqpam}, there is a non-zero element $\boldsymbol{g}=(g_1, \dots, g_r)$ of the $\bA$-linear span of the rows of $\Psi^{-1}$ such that the tractable coordinates of $\lambda$ are given by $\{g_1\!\mid_{t=\theta}, \dots, g_m\!\mid_{t=\theta}\}$ and 
    \[
 \Span_{L}(\lambda) \subseteq \Span_L \Biggl( \bigcup_{j=0}^n  \Bigl\{ \hde{j}_t(g_1)\mid_{t=\theta}, \dots, \hde{j}_t(g_m)\mid_{t=\theta}  \Bigr\} \Biggr)
    \]
    for big enough $n\geq 0$ chosen so that $(\rd\phi(t)-\theta \Id_d)^n=0$. Thus, if the tractable coordinates of periods $\lambda_1, \dots, \lambda_k$ are algebraically independent over $L$, then by Theorem~\ref{Thm:transcendence-degrees}, we conclude that all coordinates of $\lambda_1, \dots, \lambda_k$ are algebraically independent over $L$.  
\end{proof}

\begin{Example}\label{E:Drinfeld}
 \emph{Drinfeld modules.} Let $(E_\varphi, \varphi)$  be a Drinfeld module of rank $r\geq 1$ defined over a finite separable extension $L$ of $K$  from Example~\ref{E:Drinfeld1CarlitzTensor1}(i). 
By \cite[Thm.~0.2]{rp:mtcdm} and \cite[Thm.~3.5.4]{cc-mp:aipldm}, the $\fp$-adic Galois representation associated to a Drinfeld module has open image in the motivic Galois group $\Gamma_{\mot_{E_{\varphi}}}$, hence Theorem \ref{T:galois-group-of-prolongation} recovers \cite[Thm.~1.1.3]{cn:arhpldm} which determines all algebraic relations among all hyperderivatives of all periods and quasi-periods of $E_{\varphi}$.  

To explain the results more explicitly, let $\{\lambda_1, \dots, \lambda_r\}$ be an $\bA$-basis of the period lattice $\Lambda_{E_\varphi}$ and for $1 \leq j \leq r-1$, let $F_{j}(\lambda_i)$ be the quasi-periods of $E_{\varphi}$ (see Subsection~\ref{Sub:Promodulesmotives}). Then, by Theorem~\ref{T:hyperRATnPer}, one can construct a rigid analytic trivialization $\Upsilon_{E_\varphi}$ such that  
\[
 \Span_{\oK}\left( \bigcup_{\ell=0}^n\bigcup_{i=1}^r \bigcup_{j=1}^{r-1} \{\hde{n}_\theta(\lambda_i), \hde{n}_\theta(F_{j}(\lambda_i))\}\right)= \Span_{\oK}\left(\prol{n}(\Upsilon_E^{(1)}) \mid_{t=\theta} \right)
\]
 Let $\End(\varphi)\subseteq \CI$ denote the endomorphism ring  of $E_{\varphi}$, which can be identified with
 \begin{equation}\label{E:endDM}
\{c \in \CI \ | \ c\Lambda_{E_\varphi} \subseteq \Lambda_{E_\varphi}\},      
 \end{equation}
  and let $\bK_{\varphi}$ denote the field of fractions of $\End(\varphi)$. 
Let $s:=[\bK_{\varphi}:K]$, the degree of the extension. Then, by \cite[Theorem~1.2.2]{cc-mp:aipldm}, Theorem~\ref{T:galois-group-of-prolongation}, and Theorem~\ref{Thm:transcendence-degrees}, we obtain \linebreak $\trdeg_{\oK}\oK\left(\bigcup_{\ell=0}^n\bigcup_{i=1}^r \bigcup_{j=1}^{r-1} \{\hde{n}_\theta(\lambda_i), \hde{n}_\theta(F_{j}(\lambda_i))\}\right) = (n+1)r^2/s$.
\end{Example}
\begin{Example}
\emph{Tensor powers of Carlitz module.} Recall from Example~\ref{E:Drinfeld1CarlitzTensor1}(ii) that the $t$-motive $\mot_{G_k}$ associated to the $k$-th tensor power $(G_k,C^{\otimes k})$ of the Carlitz module $(E_C,C)$ is generated over $K[t]$ by $e_1=(1,0, \dots, 0)$ such that
\begin{equation}\label{E:Cntaut1}
(t-\theta)^k e_1 = \tau e_1. 
\end{equation}
Thus, $\Theta_k:=\Theta_{G_k} = (t-\theta)^k$. Consider the Anderson-Thakur power series 
\[
\Omega := \big(\sqrt[q-1]{-\theta}\big)^{-q} \prod_{i=1}^\infty \bigg(1-\dfrac{t}{\theta^{q^i}}\bigg) \in \CC_\infty \langle t \rangle,
\]
where $\sqrt[q-1]{-\theta}$ is a fixed $(q-1)$-st root of $-\theta$. 
Note that $\Upsilon_k:=(\Omega^k)^{(-1)}$ is a rigid analytic trivialization for $M_{C^{\otimes k}}$ since $\Omega^{(-1)} = (t-\theta)\Omega$.

The Galois representation associated to $G_k$ is 
\[\varrho_{G_k,\fp}=(\varrho_{E_C,\fp})^k:\Gal(K^\sep/K)\to \bA_\fp^\times, \gamma\mapsto a_\gamma^k,
\]
where
$\varrho_{E_C,\fp}$ is the Galois representation associated to the Carlitz module $E_C$. It is well known 
that 
$\varrho_{C,\fp}$ is surjective (see \cite[Proposition 12.7]{mr:ntff}). If we pick $k$ such that $p\nmid k$, the subgroup of $1$-units, $1+\fp \bA_\fp\subset \bA_\fp^\times$ is $k$-divisible, i.e., $x^k=a$ is solvable for every $a\in 1+\fp \bA_\fp$.
Therefore, if $p\nmid k$, the subgroup of $1$-units is contained in the image of $(\varrho_{C,\fp})^n$. In particular, the image is $\fp$-adically open in $\bA_\fp^\times$. Using Theorem~\ref{T:hyperRATnPer} and Theorem~\ref{Thm:transcendence-degrees}, we therefore recover \cite[Theorem 8.1]{am:ptmaip} that the coordinates of a fundamental period of $(G_k,C^{\otimes k})$ are algebraically independent if $p$ does not divide $k$. In particular, by Theorem~\ref{T:hyperRATnPer} for any $n\geq 0$ the set $\{\widetilde{\pi}^k, \hde{1}_{\theta}(\widetilde{\pi}^k), \dots, \hde{n}_{\theta}(\widetilde{\pi}^k)\}$, where $\widetilde{\pi}$ is the fundamental period of the Carlitz module $(E_C,C)$,  is algebraically independent over $\oK$ if $p$ does not divide $k$. For the case of the Carlitz module which is the case $k=1$, the fundamental period $\widetilde{\pi}$ is hypertranscendental i.e., for any $n\geq 0$ the set $\{\widetilde{\pi}, \hde{1}_{\theta}(\widetilde{\pi}), \dots, \hde{n}_{\theta}(\widetilde{\pi})\}$ is algebraically independent over $\oK$, recovering \cite[Theorem 2.1]{am:aicph}.
\end{Example}

\section{Several Drinfeld modules}\label{S:non-isogenous-Drinfeld-modules}
Our method and result are quite general and powerful.
In Subsection~\ref{SubS:nDMs}, we illustrate this by showing that under certain hypothesis, the periods and quasi-periods, and their hyperderivatives of one Drinfeld module are algebraically independent from those of other Drinfeld modules, except for some imminent relations. In Subsection~\ref{SubS:examples}, we will provide explicit families of such Drinfeld modules to illustrate how our result can be applied. As before, let $L$ be a finite separable extension of $K=\Fq(\theta)$.

\subsection{Algebraic relations among periods}\label{SubS:nDMs}

 Consider $m\geq 2$ Drinfeld modules $(E_1,\varphi_1), \ldots, (E_m,\varphi_m)$ over $L$, and let $r_1,\ldots, r_m$ be their respective ranks.
Further, for each $i=1,\ldots, m$, let $\mathbf{K}_i\subseteq \CI$ denote the field of fractions of the endomorphism ring of $\varphi_i$ as given in \eqref{E:endDM}, and let $s_i:=[\mathbf{K}_i:K]$, the degree of the extension.

In order to simplify the computations, we assume that all these Drinfeld modules are normalized so that for each $i=1,\ldots, m$, the top coefficient of $\varphi_i(t)$ (i.e.,~the coefficient of $\tau^{r_i}$ in $\varphi_i(t)$) is $(-1)^{r_i-1}$. This can be obtained by replacing $L$ by a suitable finite separable extension and by replacing $\varphi_i$ by an isomorphic one over this finite extension. 
Due to this normalization, the top exterior power $\wedge^{r_i} \mot_{E_i}$ of the corresponding $t$-motive $\mot_{E_i}$ with respect to $L[t]$ is such that $\wedge_{L[t]}^{r_i} \mot_{E_i} \cong \mot_{E_C}$, where $\mot_{E_C}$ is the $t$-motive associated to the Carlitz module $(E_C,C)$. 
Replacing Drinfeld modules by isomorphic ones does not change the dimensions of Galois groups and algebraic relations. Moreover,  the numbers $r_i$ and $s_i$ defined above remain unchanged.

\medskip

For $0\neq \mathfrak{a} \in \bA$, recall the finite kernel 
$\varphi_i[\mathfrak{a}]:= \ker(\varphi_{i}(\mathfrak{a})) = \{ \,\boldsymbol{x} \in \CC_\infty \, \mid \, \varphi_i(\mathfrak{a})\boldsymbol{x}=0\,\}$ which is isomorphic to $(\bA/(\mathfrak{a}))^{\oplus r_i}$ (see \cite[Thm.~7.2.1]{dt:ffa}).
For $i=1,\ldots, m$ and a non-zero prime $\fp\in \bA$, we let
\[   L_{i,\fp} := L(\varphi_i[\fp^\infty]):= \bigcup_{k\geq 1} L(\varphi_i[\fp^k]) \subseteq L^{\sep}\]
be the extension of $L$ obtained by adjoining all $\fp$-power torsions of $\varphi_i$, and let 
\[  \check{L}_{i,\fp}:= \bigvee_{j\ne i} L_{j,\fp} \subseteq L^{\sep}\]
denote the compositum of the fields $L_{j,\fp}$ inside $L^{\sep}$. We further let
\[ L_{C,\fp}:= L(C[\fp^\infty]):=\bigcup_{k\geq 1} L(C[\fp^k]) \subseteq L^{\sep}\]
be the extension of $L$ obtained by adjoining all  $\fp$-power torsions of the Carlitz module. We emphasize that due to our normalization, each field $L_{i,\fp}$ for all $i=1, \dots, m$, contains the field $L_{C,\fp}$.

In what follows, we will assume the following hypothesis.

\begin{Hypothesis}\label{hypo:severalDMs}
For Drinfeld modules $(E_1,\varphi_1), \ldots, (E_m,\varphi_m)$ over $L$, we will assume that they are normalized as given above. We assume, in addition, that there exists a non-zero prime $\fp \in \bA$ such that the extensions $L_{1,\fp},\ldots, L_{m,\fp}$ are linearly disjoint field extensions of $L_{C,\fp}$, i.e., inside $L^\sep$ we have $L_{i,\fp}\cap \check{L}_{i,\fp}=L_{C,\fp}$ for all $i=1,\ldots, m$.
\end{Hypothesis}

\begin{Remark}
If two of the Drinfeld modules $(E_i,\varphi_i)$ and $(E_j,\varphi_j)$ are isogenous over some extension of $L$, then Hypothesis~\ref{hypo:severalDMs} is not fulfilled. It is not clear to the authors if this is the only obstruction. 
\end{Remark}

\begin{Theorem}\label{T:severalDMs}
Let $(E_1,\varphi_1), \ldots, (E_m,\varphi_m)$ be Drinfeld modules over $L$ satisfying Hypothesis \ref{hypo:severalDMs}.
Furthermore, let $(E, \psi)=\bigoplus_{i=1}^m (E_i,\varphi_i)$ be the $t$-module of dimension $m$ defined over $L$ given by the direct sum of these Drinfeld modules, i.e., where $\psi(t)$ is given by
\[
\psi(t) = \begin{pmatrix} \varphi_1(t) & & \\ & \ddots &  \\ & & \varphi_m(t) \end{pmatrix}.
\]
Then, we have the following:
\begin{enumerate}
    \item \label{item:fibre-product} The motivic Galois group $\Gamma_{\mot_E}$ of $E$ is the fibre product of the motivic Galois groups $\Gamma_{\mot_{E_i}}$ over the motivic Galois group $\Gamma_{\mot_{C}}$, and the image of the $\fp$-adic Galois representation of $E$ is $\fp$-adically open in $\Gamma_{\mot_E}$.
    \item \label{item:dimension-formula} For all $n\geq 0$, we have $\Gamma_{\rho_{n}\mot_E}=\rP{n}(\Gamma_{\mot_E})$ and 
    \begin{equation}\label{E:dimproDMeg}
\dim(\Gamma_{\rho_{n}\mot_E}) = (n+1)\cdot \left( 1 + \sum_{i=1}^m \Bigl(\frac{r_i^2}{s_i} -1 \Bigr) \right) 
    \end{equation} 
    \item \label{item:fibre-product-prolong} For all $m\geq 0$, the motivic Galois group $\Gamma_{\rho_{m}\mot_E}$  is the fibre product of the motivic Galois groups $\Gamma_{\rho_{m}\mot_{E_i}}$ over the motivic Galois group $\Gamma_{\rho_{m}\mot_{C}}$.    
\end{enumerate}
\end{Theorem}

\begin{proof}
We start by showing \eqref{item:fibre-product}.
Since $E$ is the direct sum of all $E_i$, its $t$-motive is $\mot_E = \bigoplus_{i=1}^m\mot_{E_i}$, the direct sum of the $t$-motives $\mot_{E_i}$. 
Therefore, we have an embedding of motivic Galois groups
\begin{equation}\label{eq:embedding-in-direct-product}
\Gamma_{\mot_E}\subseteq \prod_{i=1}^m \Gamma_{\mot_{E_i}}.\end{equation}
Since, for all $i=1, \dots, m$, we have 
$\wedge^{r_i} \mot_{E_i} = \mot_C$, we see that $\mot_C$ is in the Tannakian category generated by $\mot_{E_i}$ and so, we have surjective homomorphisms $\Gamma_{\mot_{E_i}} \twoheadrightarrow\Gamma_{\mot_C}$, induced by restricting to the Tannakian subcategory generated by $\mot_C$.
Thus, the embedding \eqref{eq:embedding-in-direct-product} factors through the fibre product of all $\Gamma_{\mot_{E_i}}$ over $\Gamma_{\mot_{C}}$, which we will denote by $\Gamma'$.
In the standard embeddings of $\Gamma_{\mot_{E_i}}\subseteq \GL_{r_i}$ and $\Gamma_{\mot_C}\cong \Gm$, the homomorphisms $\Gamma_{\mot_{E_i}} \twoheadrightarrow \Gamma_{\mot_C}$ are given by taking determinants, and hence, we have
\[ \Gamma_{\mot_E} \subseteq \Gamma':= \left\{ \begin{pmatrix}
    D_1 & & & \\ & D_2 & & \\ & & \ddots & \\ & & & D_m
\end{pmatrix} \,\,\middle|\,\, D_i\in \Gamma_{\mot_{E_i}}\subseteq \GL_{r_i}, \det(D_1)=\cdots = \det(D_m)  \right\}.  \]
We will show that the image of the Galois representation $\varrho_{E,\fp}$ is $\fp$-adically open in the group $\Gamma'(\bK_\fp)$. Then, since $\Gamma'$ is connected, this shows that the inclusion $\Gamma_{\mot_E} \subseteq \Gamma'$ is, in fact, an equality. 
This proves \eqref{item:fibre-product}.

\medskip

Since the Galois extensions $L_{1,\fp},\ldots, L_{m,\fp}$ are linearly disjoint field extensions of $L_{C,\fp}$, we have 
\[  \Gal\Big(\bigvee_{i} L_{i,\fp}/ L_{C,\fp}\Big) \cong \Gal(L_{1,\fp}/L_{C,\fp})\times \dots \times \Gal(L_{m,\fp}/L_{C,\fp}), \]
where $\bigvee_{i} L_{i,\fp}$ denotes the compositum of the fields $L_{i, \fp}$, $i=1, \dots, m$ inside $L^{\sep}$.
Recall that for all $i=1, \dots, n$, the images of the Galois representations 
\[ \varrho_{E_i,\fp}(\Gal(L^\sep/L))=\varrho_{E_i,\fp}(\Gal(L_{i,\fp}/L))\subseteq \Gamma_{\mot_{E_i}}(\bK_\fp)\]
are $\fp$-adically open \cite[Thm.~0.2]{rp:mtcdm}, and note that $\det(\varrho_{E_i,\fp}(\gamma))=1$ if and only if $\gamma\in \Gal(L^\sep/L)$ fixes $L_{C,\fp}$. 
Thus, we see that the image of 
\begin{align*}
    &\Gal\Big(\bigvee_{i} L_{i,\fp}/ L_{C,\fp}\Big) \cong \Gal(L_{1,\fp}/L_{C,\fp})\times \dots \times \Gal(L_{m,\fp}/L_{C,\fp})
    \hookrightarrow \ker\Bigl(\Gamma'(\bK_{\fp})\to \Gamma_{\mot_C}(\bK_{\fp})\Bigr)
\end{align*} 
is also $\fp$-adically open. Finally, the image of $\Gal(L_{C,\fp}/L)$ in $\Gamma_{\mot_C}(\bK_{\fp})$ is  isomorphic to $\Gm(\bA_\fp)$, and hence the  Galois representation $\varrho_{E,\fp}$ has a $\fp$-adically open image in $\Gamma'$.

\medskip

Next, we prove part \eqref{item:dimension-formula}. By part \eqref{item:fibre-product} and Theorem \ref{T:galois-group-of-prolongation}, we have $\Gamma_{\rho_{n}\mot_E}=\rP{n}(\Gamma_{\mot_E})$, and
$\dim(\Gamma_{\rho_{n}\mot_E})=(n+1)\cdot \dim(\Gamma_{\mot_E})$. So it only remains to show the formula in \eqref{E:dimproDMeg} for $n=0$.
Since $\dim(\Gamma_{\mot_C})=1$, $\dim(\Gamma_{\mot_{E_i}})=\frac{r_i^2}{s_i}$ for all $i=1,\ldots, m$, and  by part \eqref{item:fibre-product}, $\Gamma_{\mot_E}$ is the fibre product of all $\Gamma_{\mot_{E_i}}$ over $\Gamma_{\mot_{C}}$, the dimension formula for fibre products gives
\[ \dim(\Gamma_{\mot_E})=\sum_{i=1}^m \dim(\Gamma_{\mot_{E_i}}) - (m-1)\cdot \dim(\Gamma_{\mot_C}) =  1 + \sum_{i=1}^m \Bigl(\frac{r_i^2}{s_i} -1 \Bigr) , \]
which is the claim for $n=0$.

For part \eqref{item:fibre-product-prolong}, we first recall that by Theorem~\ref{T:galois-group-of-prolongation}, we have  $\Gamma_{\rho_{n}\mot_C}=\rP{n}(\Gamma_{\mot_C})$ and $\Gamma_{\rho_{n}\mot_{E_i}}=\rP{n}(\Gamma_{\mot_{E_i}})$ for all $i=1,\ldots, m$.
Then by Lemma \ref{lem:prolongations-preserve-fibre-products}, we get
\begin{align*}
    \Gamma_{\rho_{n}\mot_E} &\cong 
    \rP{n}(\Gamma_{\mot_E}) \cong \rP{n}(\Gamma_{\mot_{E_1}}\times_{\Gamma_{\mot_C}}\cdots \times_{\Gamma_{\mot_C}} \Gamma_{\mot_{E_m}} ) \\
    & \cong \rP{n}(\Gamma_{\mot_{E_1}})\times_{\rP{n}(\Gamma_{\mot_C})} \cdots \times_{\rP{n}(\Gamma_{\mot_C})} \rP{n}(\Gamma_{\mot_{E_m}})\\ & \cong \Gamma_{\rho_{n}\mot_{E_1}}\times_{\Gamma_{\rho_{n}\mot_C}}\cdots \times_{\Gamma_{\rho_{n}\mot_C}} \Gamma_{\rho_{n}\mot_{E_m}} 
\end{align*}  
\end{proof}

The following corollary follows from  Theorems~\ref{T:Pap} and 
\ref{T:severalDMs}. 
\begin{Corollary}
  Let $(E_1,\varphi_1), \ldots, (E_n,\varphi_n)$ be Drinfeld modules over $L$ satisfying Hypothesis \ref{hypo:severalDMs}. 
  Then, 
the only algebraic relations between the entries of all 
$\Upsilon_{E_i}^{(1)} |_{t=\theta}$ ($i=1,\ldots, n$), and their hyperderivatives, are the algebraic relations among the entries of each 
$\Upsilon_{E_i}^{(1)} |_{t=\theta}$,  the relation $\det(\Upsilon_{E_1}^{(1)} |_{t=\theta})=\det(\Upsilon_{E_2}^{(1)} |_{t=\theta})=\dots = \det(\Upsilon_{E_m}^{(1)} |_{t=\theta})$, and the relations derived from these by hyperdifferentiation. That is, for any $n \geq 0$
\begin{align*}
\trdeg_{\oK}\oK(\Upsilon_{\rho_{n} E}^{(1)}|_{t=\theta}) &= (n+1)\cdot \trdeg_{\oK}\oK(\Upsilon_{E}^{(1)} |_{t=\theta})\\
&= (n+1)\cdot  \left( 1 + \sum_{i=1}^m \Bigl(\frac{r_i^2}{s_i} -1 \Bigr) \right)
\end{align*}
\end{Corollary}

 Combining Theorem~\ref{T:hyperRATnPer}, Example~\ref{E:Drinfeld}, and this corollary, we obtain the following. 

\begin{Corollary}\label{C:trdegSeveralDM} 
    Let $(E_1,\varphi_1), \ldots, (E_n,\varphi_n)$ be Drinfeld modules over $L$ satisfying Hypothesis \ref{hypo:severalDMs}.  
    Then, the only algebraic relations among all the periods, quasi-periods, and their hyperderivatives are the algebraic relations among the periods and quasi-periods of each $E_i$, $i=1,\dots, m$, Anderson's Legendre relation, and the relations derived thereof by hyperdifferentiating.
    
 In particular, the field extension $\mathcal{F}_n$ generated over $L$ by all the periods, quasi-periods and their hyperderivatives up to order $n$, satisfies 
    \[ \trdeg(\mathcal{F}_n:L) = (n+1)\cdot \left( 1 + \sum_{i=1}^m \Bigl(\frac{r_i^2}{s_i} -1 \Bigr) \right). \]
\end{Corollary}

\subsection{Family of examples}\label{SubS:examples}

In this section, we illustrate Theorem \ref{T:severalDMs} by constructing a family of examples satisfying Hypothesis \ref{hypo:severalDMs}.
In this family of examples, we restrict ourselves to two Drinfeld modules. 

\begin{Theorem}
Let $(E_1, \varphi_1)$ be given by $\varphi_1(t)=\theta+a\tau-\tau^2$, and $(E_2, \varphi_2)$ be given by $\varphi_2(t)=\theta+b\tau-\tau^2$ over $K=\Fq(\theta)$ with $a,b\in K$ such that the following conditions hold:
\begin{itemize}
\item There is a prime $\ell\subseteq \Fq[\theta]$, $\ell\neq (\theta)$ such that $v_\ell(a)\geq 0$, $v_\ell(b)<0$, and $\operatorname{gcd}(q, v_\ell(b))=1$,
\item $\Gal(K(\varphi_2[t])/K)\supseteq \SL_2(\Fq)$,
\end{itemize}
Here $v_\ell$ denotes the $\ell$-adic valuation. For $i=1,2$, let  
$K_{i}:=\bigcup_{k\geq 1} K(\varphi_i[t^k])$, and let $ K_{C}:= \bigcup_{k\geq 1} K(C[t^k])$. 
Then, for the $t$-adic Galois representation $\varrho_{E_2,t}$, we have
\[ \varrho_{E_2,t}\bigl( \Gal(K_{2}/K_{C}) \bigr) = \SL_2(\Fq\ps{t}), \]
and $K_{1}$ and $K_{2}$ are linearly disjoint over $K_{C}$.

In particular, $(E_1,\varphi_1)$ and $(E_2, \varphi_2)$ satisfy Hypothesis \ref{hypo:severalDMs} for $\fp=(t)$.
\end{Theorem}

\begin{proof}
The proof is done in several steps where we compute the Galois group of $K_2$ over $K_C$ and  the Galois group of the compositum $K_1\cdot K_2$ over $K_1$.\\

\paragraph{Claim 1:} \emph{With respect to an appropriate $\Fq\ps{t}$-basis of $\varphi_2[t^\infty]=\bigcup_{k\geq 0}\varphi_2[t^k]$, the inertia group $I_\ell$ of $K_2/K$ at $\ell$ contains 
the group
\[  I:= \left\{ \begin{pmatrix} 1 & \beta \\ 0 & 1 \end{pmatrix} \middle| \beta\in \Fq\ps{t} \right\} \]
of upper triangular matrices with coefficients in $\Fq\ps{t}$.}

\medskip

Let $m:=\operatorname{gcd}(q-1,v_\ell(b))$, and take $\tilde{K}$ to be an extension of $K$ totally ramified at $\ell$ of degree $\frac{q-1}{m}$, and let $\tilde{\ell}$ be the place in $\tilde{K}$ lying over $\ell$. 
Then the Drinfeld module $\varphi_2$ considered over the completion $\tilde{K}_{\tilde{\ell}}$ satisfies the hypotheses of \cite[Prop.~4.1]{dz:dmmgatp}, for $\mathfrak{a}=t^n$ (any $n\geq 0$).
Furthermore, the top exterior power of  $\varphi_2$ is isomorphic to the Carlitz module (i.e., $\wedge^{2} \mot_{E_2} \cong M_{E_C}$), 
and the $j$-invariant of $\varphi_2$ is $j(\varphi_2)=\frac{b^{q+1}}{-1}$, and satisfies $p \nmid v_\ell(j(\varphi_2))=(q+1)\cdot v_\ell(b)$ by assumption on $b$.

Hence, the proof of \cite[Prop.4.4]{dz:dmmgatp} shows that for any $n\geq 0$, the inertia group of the $t^n$-torsion of $\varphi_2$ over $\tilde{K}_{\tilde{\ell}}$ equals
\[   \left\{ \begin{pmatrix} 1 & \beta \\ 0 & 1 \end{pmatrix} \middle| \beta\in \Fq[t]/(t^n) \right\}. \] 
Therefore, the inertia group $I_\ell$ over the smaller field $K$ contains that group. The claim is obtained by going to the limit $n\to \infty$.\\

\paragraph{Claim 2:} \emph{The Galois group $\Gal(K_1\cdot K_2, K_1)$ contains the group $I$ from Claim 1.}

\medskip

Since $v_\ell(a)\geq 0$, the Drinfeld module $\varphi_1$ has good reduction at $\ell$. By \cite[Theorem 6.3.1]{mp:dm}, $K_1/K$ is unramified at $\ell$. Hence, the inertia group at $\ell$ (more precisely at a place above $\ell$) of $K_1\cdot K_2/K_1$ is the same as that of $K_2/K$.\\

\paragraph{Claim 3:} \emph{$\Gal(K_1(\varphi_2[t])/K_1)=\SL_2(\Fq)=\Gal(K(\varphi_2[t])/K(C[t]))$.}

\medskip

As $\Gal(K(\varphi_2[t])/K)\supseteq \SL_2(\Fq)$, and the restriction $\Gal(K(\varphi_2[t])/K)\to \Gal(K(C[t])/K)$ corresponds to the determinant homomorphism, we have $\Gal(K(\varphi_2[t])/K(C[t]))=\SL_2(\Fq)$.

Therefore $\Gal(K_1(\varphi_2[t])/K_1)\cong \Gal(K(\varphi_2[t])/K_1\cap K(\varphi_2[t]))$ is a subgroup of $\SL_2(\Fq)$, and even a normal subgroup, since $K_1\cap K(\varphi_2[t])$ is Galois over $K$. By Claim 2, $\Gal(K_1(\varphi_2[t])/K_1)$ contains all upper triangular matrices. As there is no strict normal subgroup of $\SL_2(\Fq)$ containing the upper triangular matrices, we have $\Gal(K_1(\varphi_2[t])/K_1)=\SL_2(\Fq)$.\\

\paragraph{Claim 4:} \emph{$\Gal(K_1\cdot K_2, K_1)=\SL_2(\Fq\ps{t})=\Gal(K_2/K_C)$.} 

\medskip

By construction, both Galois groups are contained in $\SL_2(\Fq\ps{t})$, and as in Claim 3, $\Gal(K_1\cdot K_2, K_1)$ is a normal subgroup of $\Gal(K_2/K_C)$. So it suffices to show that $\Gal(K_1\cdot K_2, K_1)=\SL_2(\Fq\ps{t})$.
By Claim 2, $\Gal(K_1\cdot K_2, K_1)$ contains the subgroup $I$ of upper triangular matrices, and by Claim 3, the composition $\Gal(K_1\cdot K_2, K_1)\to \SL_2(\Fq\ps{t})\twoheadrightarrow \SL_2(\Fq)$ is surjective.
Let $G$ be the image of $\Gal(K_1\cdot K_2, K_1)$ in $\SL_2(\Fq\ps{t})$, and denote for $n\geq 1$,
\[ \pi_n: \SL_2(\Fq\ps{t})\longrightarrow \SL_2(\Fq[t]/(t^n)),\]
as well as 
\[\Gamma_n:=\Ker(\pi_n). \]
We will show by induction that $\pi_n(G)=\SL_2(\Fq[t]/(t^n))$ for all $n$, and hence $G=\SL_2(\Fq\ps{t})$.

The case $n=1$ was given by Claim 3. So assume that $\pi_n(G)=\SL_2(\Fq[t]/(t^n))$ for some $n$.

By definition, we have a short exact sequence of groups
\[ 1 \to \pi_{n+1}(\Gamma_n) \to \SL_2(\Fq[t]/(t^{n+1}) \to \SL_2(\Fq[t]/(t^n)) \to 1, \]
where
\[ \pi_{n+1}(\Gamma_n)=\{ 1+t^nM\in \SL_2(\Fq[t]/(t^{n+1}) \mid M\in \Mat_{2\times 2}(\Fq) \} \cong \{ M\in \Mat_{2\times 2}(\Fq) \mid \tr(M)=0\}.\]
As this is a commutative group, an element $D\in \SL_2(\Fq[t]/(t^n))$ acts on $\pi_{n+1}(\Gamma_n)$ via conjugation by an arbitrary representative in $\SL_2(\Fq[t]/(t^{n+1})$ which turns out to be just the conjugation operation of $\pi_1(D)$ on $H:=\{ M\in \Mat_{2\times 2}(\Fq) \mid \tr(M)=0\}$.

As $\pi_1(G)=\SL_2(\Fq)$, the intersection $\pi_{n+1}(\Gamma_n)\cap \pi_{n+1}(G)$ is therefore a subgroup of $H$ stable under conjugation by $\SL_2(\Fq)$.
Further by Claim 2, it contains the subgroup of strict upper triangular matrices $\begin{pmatrix}
0 & b \\ 0 & 0
\end{pmatrix}$.

An easy calculation shows that this implies $\pi_{n+1}(\Gamma_n)\cap \pi_{n+1}(G)=H$. Combining this with $\pi_n(G)=\SL_2(\Fq[t]/(t^n))$ shows 
$\pi_{n+1}(G)=\SL_2(\Fq[t]/(t^{n+1}))$.
\end{proof}

From the calculations in the previous proof, we get a sufficient criterion for a surjective $t$-adic Galois representation.

\begin{Corollary}\label{cor:surjective-Galois-representation}
Let $(E_2, \varphi_2)$ be a rank $2$ Drinfeld module over  $K=\Fq(\theta)$ given by $\varphi_2(t)=\theta+b\tau-\tau^2$ with $b\in K$ such that the following conditions hold:
\begin{itemize}
\item There is a prime $\ell\subseteq \Fq[\theta]$, $\ell\neq (\theta)$ such that $v_\ell(b)<0$, and $\operatorname{gcd}(q, v_\ell(b))=1$,
\item $\Gal(K(\varphi_2[t])/K)\supseteq \SL_2(\Fq)$,
\end{itemize}
Then the $t$-adic Galois representation associated to $E_2$ is surjective, i.e.~$\Gal(K_2/K)\cong \GL_2(\Fq\ps{t})$.
\end{Corollary}

\begin{proof}
    We showed above (Claim 4) that $\Gal(K_2/K_C)\hookrightarrow \SL_2(\Fq\ps{t})$ is an isomorphism. So we have the following commutative diagram with short exact rows where the first and last vertical arrows are isomorphisms:
\[\xymatrix{
1\ar[r] & \Gal(K_2/K_C) \ar[r] \ar[d]^{\cong} & \Gal(K_2/K) \ar[r]\ar@{^{(}->}[d] & \Gal(K_C/K) \ar[r] \ar[d]^{\cong} & 1 \\
1\ar[r] &  \SL_2(\Fq\ps{t})\ar[r] & \GL_2(\Fq\ps{t}) \ar[r]^{\det} & \GG_m(\Fq\ps{t}) \ar[r] & 1. 
}\]
By the short 5-lemma, the middle vertical arrow also has to be an isomorphism.
\end{proof}

One way to obtain $\Gal(K(\varphi_2[t])/K)\supseteq \SL_2(\Fq)$ with the given assumptions on $b$ is given in the following proposition.

\begin{Proposition}
Let $(E_2, \varphi_2)$ be a rank $2$ Drinfeld module over  $K=\Fq(\theta)$ given by $\varphi_2(t)=\theta+b\tau-\tau^2$ with $b\in K$, and let $\ell\subseteq \Fq[\theta]$ be a prime ideal such that the following conditions hold:
\begin{enumerate}[(i)]
\item the polynomial $x^{q-1}+\theta$ is irreducible modulo $\ell$ (i.e.,~in $\FF_\ell[x]=(\Fq[\theta]/\ell)[x]$),
\item $v_\ell(b)<0$ is divisible by $q-1$, and $\operatorname{gcd}(q, v_\ell(b))=1$,
\item $-b_0^{-1}\theta$ is a $(q-1)$-st power modulo $\ell$ (i.e.,~in the field $\FF_\ell=\Fq[\theta]/\ell$).
\item $v_\infty(b)\geq 0$.
\end{enumerate}
Here, for a uniformizer $u_\ell$ at the prime $\ell$, $b_0$ is the leading coefficient of the $\ell$-adic expansion $b=b_0u_\ell^{v_\ell(b)}+b_1u_\ell^{v_\ell(b)+1}+\ldots$.\footnote{Although at first sight, it seems that condition (iii) depends on the choice of the uniformizer $u_\ell$, the assumption that $v_\ell(b)$ is divisible by $q-1$ in condition (ii) ensures that it does not depend on that choice.}

Then, $\Gal(K(\varphi_2[t])/K)=\GL_2(\Fq)$,
\end{Proposition}

\begin{proof}
At the infinite place, the Newton polygon of the polynomial $f(x)=x^{q^2}-bx^q-\theta x$ -- whose roots are the elements in $\varphi_2[t]$ --  has one edge of slope $\frac{1}{q^2-1}$ due to the condition $v_\infty(b)\geq 0$. Therefore, $f(x)/x=x^{q^2-1}-bx^{q-1}-\theta$ is irreducible over $K$. In particular, $\Gal(K(\varphi_2[t])/K)$ acts transitively on $\varphi_2[t]\setminus \{0\}$.

We claim that the decomposition group $G_\ell$ of $K(\varphi_2[t])/K$ at a prime above $\ell$ equals
\[ \left\{ \begin{pmatrix} 1 & \beta \\ 0 & \gamma \end{pmatrix} \middle| \gamma\in \Fq^\times, \beta\in \Fq \right\} \]
with respect to an appropriate $\Fq$-basis $\{\zeta_1, \zeta_2\}$ of $\varphi_2[t]$.

Then the stabilizer of the first basis vector $\zeta_1$ has order $q^2-q$. Therefore, the order of the Galois group is at least $(q^2-1)\cdot (q^2-q)$ which is exactly the order of $\GL_2(\Fq)$.
Hence, $\Gal(K(\varphi_2[t])/K)\cong \GL_2(\Fq)$.    

\medskip

It remains to show that 
\[ G_\ell = \left\{ \begin{pmatrix} 1 & \beta \\ 0 & \gamma \end{pmatrix} \middle| \gamma\in \Fq^\times, \beta\in \Fq \right\}. \]

The extension $K(\varphi_2[t])/K$ is ramified at $\ell$, as the Newton polygon\footnote{Now we consider the Newton polygon with respect to the $\ell$-adic valuation $v_\ell$.} of the polynomial $f(x)=x^{q^2}-bx^q-\theta x\,$ has slopes $-\frac{v_\ell(b)}{q^2-q}\in \frac{1}{q}\ZZ\setminus \ZZ$ and $\frac{v_\ell(b)}{q-1}\in \ZZ$, by assumption on $b$. Hence, with respect to the $\Fq$-basis of $\varphi_2[t]$ given by a ``short'' root $\zeta$ with valuation $-\frac{v_\ell(b)}{q-1}>0$ and a ``long'' root with valuation $\frac{v_\ell(b)}{q^2-q}<0$, we have
\[ G_\ell \subseteq \left\{ \begin{pmatrix} \alpha & \beta \\ 0 & \gamma \end{pmatrix} \middle| \alpha,\gamma\in \Fq^\times, \beta\in \Fq \right\}, \]
since the decomposition group preserves valuations.

Let $v=-\frac{v_\ell(b)}{q-1}$, and $s=u_\ell^v$, so that one has $v_\ell(s^{-1}\zeta)=0$. By considering the polynomial $s^{-1}f(sx)$ modulo $\ell$, one computes that 
\[  \prod_{\mu\in \Fq} (x-\mu s^{-1}\zeta) = x^q-(s^{-1}\zeta)^{q-1}x \equiv x^q - (-b_0^{-1}\theta)x \mod{\ell}. \]
By assumption (iii), there is $c\in \FF_\ell$ with \[ c^{q-1}\equiv -b_0^{-1}\theta \mod{\ell},\]
i.e., $x^q-(-b_0^{-1}\theta)x$ factors completely modulo $\ell$. So  by Hensel's lemma the polynomial factors completely over $K_\ell$ which implies that $\zeta\in K_\ell$. Hence indeed,
\[ G_\ell \subseteq \left\{ \begin{pmatrix} 1 & \beta \\ 0 & \gamma \end{pmatrix} \middle| \gamma\in \Fq^\times, \beta\in \Fq \right\}. \]

A short calculation in $K_\ell[x]$ shows that
\[ f(x)/(x^q-\zeta^{q-1} x) = \left(x^q-\zeta^{q-1}x\right)^{q-1} +\theta \zeta^{1-q} = \zeta^{1-q}\left( \left( \zeta x^q-\zeta^q x \right)^{q-1}   +\theta \right).  \]
If the latter polynomial would factor, it would factor into $q-1$ factors of order $q$ due to the slope of its Newton polygon, i.e., as
\[ \zeta^{1-q}\left( \left( \zeta x^q-\zeta^q x \right)^{q-1}   +\theta \right) =\zeta^{1-q} \prod_{\mu\in \Fq^\times} \left( \zeta x^q-\zeta^q x -\mu \eta_\theta \right),    \]
where $\eta_\theta$ is a $(q-1)$-st root of $(-\theta)$ in $K_\ell$.
By assumption (i), such a root does not exist modulo $\ell$, and hence not in $K_\ell$.

Hence, indeed
\begin{equation*}
    G_\ell = \left\{ \begin{pmatrix} 1 & \beta \\ 0 & \gamma \end{pmatrix} \middle| \gamma\in \Fq^\times, \beta\in \Fq \right\}. \qedhere
\end{equation*} 
\end{proof}

\begin{Remark}
    For similar but unrelated constructions as we have done in this section, see e.g. \cite{ar:tgrspddm, dz:dmmgatp,  nk-ds:ostgrdmr, ld-jf:grptdmgc}.
    
\end{Remark}

\def\cprime{$'$}

\end{document}